\documentclass[12pt,reqno]{amsart}
\usepackage{enumerate}
\usepackage{amsmath, amsfonts, amssymb, amsthm}
\def\CC{\mathcal C}
\def\F{{\mathbf F}}
\def\HH{\mathcal H}
\def\II{\mathcal I}
\def\llambda{{\pmb\lambda}}
\def\MM{\mathcal M}
\def\PP{\mathcal P}
\def\xx{{\mathbf x}}
\def\sss{{\mathbf s}}
\def\Ii{{\mathcal I}}
\def\Llambda{{\boldsymbol \Lambda}}

\def\GL{\operatorname{GL}}

\def\rank{\operatorname{rank}}   %rank
\def\Sp{\operatorname{Sp}}

\newcommand{\abs}[1]{|#1|}      %absolute value

\numberwithin{equation}{section}
\newtheorem{Theorem} {Theorem} [section]

\newtheorem{Lemma} [Theorem] {Lemma}

\theoremstyle{definition}
\newtheorem{Example}[Theorem]{Example}
\newtheorem{Definition}[Theorem]{Definition}
\newtheorem{Remark}[Theorem]{Remark}

\begin{document}
\title
{Incidence Modules for Symplectic Spaces in Characteristic Two}
\author[Chandler, Sin, and Xiang]{David B. Chandler$^{\dagger}$, Peter Sin, Qing Xiang$^{\ddagger}$}

\thanks{$^{\dagger}$This author worked at Institute of Mathematics, Academia Sinica, Taipei,
Taiwan}
\thanks{$^{\ddagger}$Research supported in part by NSF Grant DMS 0701049.}
\address{%Institute of Mathematics, Academia Sinica,
%Nangang, Taipei 11529,  Taiwan} \email{chandler@math.sinica.edu.tw}
6 Georgian Circle, Newark, DE  19711} \email{davidbchandler@gmail.com}
\address{Department of Mathematics, University of Florida, Gainesville, FL 32611,
USA} \email{sin@math.ufl.edu}
\address{Department of Mathematical Sciences, University of Delaware, Newark, DE 19716, USA} \email{xiang@math.udel.edu}

 %\subjclass{ 05E20 (Primary), 20G05 20C11 (Secondary)}
\begin{abstract}
We study the permutation action of a finite symplectic group
of characteristic $2$ on the set of subspaces
of its standard module which are either totally isotropic
or else complementary to totally isotropic subspaces with respect to the
alternating form.
A general formula is obtained for the $2$-rank of the incidence
matrix for the inclusion of one-dimensional subspaces in
the distinguished subspaces of a fixed dimension.
\end{abstract}

\maketitle

%}}}
\section{Introduction}
Let $V$ be a finite-dimensional vector space over $\F_q$, $q=2^t$.
We assume in addition that $V$ has a nonsingular alternating form
$b(\cdot,\cdot)$. Then the dimension of $V$ must be an even number
$2m$.  We fix a basis $e_1,e_2,\ldots,e_m,f_m,\ldots,f_1$ of $V$, with corresponding coordinates $x_1,x_2,\ldots,x_m,y_m,\ldots,y_1$ so that
$b(e_i,f_j)=\delta_{ij},\ b(e_i,e_j)=0,$ and $b(f_i,f_j)=0$, for all $i$ and $j$.

For any subspace $W$ of $V$, let $W^\perp=\{v\in V\mid b(v,w)=0,
\forall\ w\in W\}$ be its complement.  A subspace $W$ is called {\it totally
isotropic} if $b(\xx,\xx')=0$ for any two vectors, $\xx,\xx'\in W$.
 We will be interested in totally isotropic subspaces
of $V$; these necessarily have dimensions $r\leq m$.
Also, we will consider the complements (with respect to the
alternating form) of totally
isotropic subspaces, which of course will have dimensions $r\geq m$.

For $1\leq r\leq 2m-1$, let $\II_r=\II_r(t)$ denote the set of
totally isotropic subspaces, or complements of such, of dimension
$r$. Let $B_{r,1}=B_{r,1}(t)$ denote the $(0,1)$-incidence matrix of
the natural inclusion relation between $\II_1$ and $\II_r$. Its rows
are indexed by $\II_r$ and its columns by $\II_1$, with the entry
corresponding to $Y\in\II_1$ and $X\in\II_r$ equal to $1$ if and
only if $Y$ is contained in $X$. We now state a theorem giving the
{\it $2$-ranks} of the matrices $B_{r,1}$, that is, their ranks when
the entries are considered as elements of $\F_2$. In order to
simplify this and several other statements we will employ the
notational convention that
 $\delta(P)=1$ if a statement $P$ holds, and $\delta(P)=0$ otherwise.

\begin{Theorem}
\label{rankthm}
Let $m\geq 2$ and $1\leq r\leq 2m-1$.
Let $A$ be the $(2m-r)\times(2m-r)$-matrix whose $(i,j)$-entry is
$$a_{i,j}=
{2m \choose 2j-i}-{2m \choose 2j+i+2r-4m-2- 2(m-r)\delta(r\leq m)}.$$
Then
$$
\rank_2 (B_{r,1}(t))= 1+{\rm Trace}(A^t).
$$
\end{Theorem}

The significance of the entries $a_{i,j}$ is that they are the dimensions of
certain representations of the symplectic group $\Sp(V)$ which
are restrictions of representations of the algebraic group $\Sp(2m,\overline{\F}_q)$,
where $\overline{\F}_q$ is an algebraic closure of $\F_q$.

\begin{Example}
When $m=r=2$, the matrix $A$ is
$$
\left(
\begin{array}{cc}
4 & 4\\
1 & 5 \end{array}\right),
$$
whose eigenvalues are $\frac{9\pm\sqrt{17}}{2}=(\frac{1\pm\sqrt{17}}{2})^2$.
Thus,
$$
\rank_2(B_{2,1}(t))=1+\left(\frac{1+\sqrt{17}}{2}\right)^{2t}+\left(
\frac{1-\sqrt{17}}{2}\right)^{2t}.
$$

This formula was previously proved in \cite{sastry-sin} by using
very detailed information about the extensions of simple modules for
$\Sp(4,q)$; such cohomological information is unavailable when $V$
has higher dimension. Moreover, in \cite{sastry-sin} the numbers
$\frac{1\pm\sqrt{17}}{2}$ arise in a rather mysterious fashion from
the combinatorics involved. Theorem~\ref{rankthm}, whose proof does
not depend on \cite{sastry-sin}, tells us that squares of these
numbers are the eigenvalues of a matrix whose entries are the
dimensions of modules for $\Sp(4,q)$. This explanation is more
natural.
\end{Example}

We denote by $k[\PP]$ the space of functions from $\PP=\II_1$ to $k=\F_q$.
The action of $\Sp(V)$ on $\PP$ makes $k[\PP]$ into a permutation module.
Since $\Sp(V)$ acts transitively on $\II_r$,
the $2$-rank of $B_{r,1}$ is equal to the dimension of the
$k\Sp(V)$-submodule of $k[\PP]$ generated by the characteristic
function of one element of $\II_r$. Let us denote this
submodule by  $\CC_r$.
%added by peter; I think it helps the reader keep track.
Since the sum of all the characteristic functions of totally isotropic
$r$-subspaces contained in a given totally  isotropic $(r+1)$-subspace
is equal to the characteristic function of the latter, with a similar
relation for complements of totally isotropic subspaces, we have

\begin{equation}\label{Cnested}
\mathcal \CC_{2m-1}\subset\cdots\subset \mathcal \CC_{2}\subset \mathcal
\CC_{1}=k[\PP].
\end{equation}

Functions in $k[\PP]$ can be written as polynomials in the coordinates.
For example, the characteristic function of the isotropic subspace $W_0$
defined by $x_1=0$, $x_2=0$, \ldots, $x_m=0$ is
\begin{equation}
\label{fav_subspace}
\chi_{{}_{W_0}}=\prod_{i=1}^m(1-x_i^{q-1}).
\end{equation}
Theorem~\ref{rankthm} is proved by relating this polynomial description
to the actions of the groups $\GL(V)$ and $\Sp(V)$ on $k[\PP]$.
We deduce it as a numerical corollary of  Theorem~\ref{basis_thm}
which gives a basis of $\CC_{r}$ of the right size.

When $q=p^t$ is odd, the $p$-rank of the corresponding incidence
matrices has been computed in all dimensions in \cite{CSX2}. It was
noted in \cite{CSX2} that for $m=2$ and $r=2$, there is a single
formula in $p$ and $t$ which gives the rank for both even and odd
values of $q$, even though there is no uniform proof. This
circumstance is somewhat coincidental, however, since the examples
in Section~\ref{examples} show that for $m=r=3$, the ranks are not
given
 by the same function of $p$ and $t$ for the even characteristic and odd characteristic cases.
The even-odd distinction
reflects some fundamental differences in the $k\Sp(V)$-submodule
structures of $k[\PP]$ between the even and odd cases,
which are apparent in  the action of
$\Sp(V)$ on the graded pieces of the truncated polynomial algebra.
In a sense which will become clear, these modules are the basic
building blocks of $k[\PP]$.
When $p$ is odd, the action of $\Sp(V)$ on each homogeneous piece of
the truncated polynomial algebra is semisimple, but for $p=2$,
the submodule structure is much richer. In characteristic $2$
the truncated polynomial algebra is an exterior algebra.
The scalar extensions of the exterior powers to the algebraic closure
$\overline k$ are examples of {\it tilting modules} for the algebraic
group $\Sp(V\otimes\overline k)$, as described in \cite{donkin};
these have filtrations by  {\it Weyl modules}  and their duals. (See \cite{jantzen_book} for definitions.)
These filtrations are very important in our work; indeed
the entries of the matrix $A$ in Theorem~\ref{rankthm} are the dimensions
of submodules which appear as terms in them.
In order to keep prerequisites to a minimum, our treatment of filtrations
of exterior powers in Section~\ref{wedges} is self-contained,  requiring no general facts
about Weyl modules, but sufficient to define the filtrations
and compute the dimensions of the subquotients.

The exterior powers are related  to $k[\PP]$ through its
$k\GL(V)$-module structure. The $k\GL(V)$-submodule lattice of
$k[\PP]$ has a simple description \cite{bardoe-sin} in terms of a
partially ordered set $\HH$ of certain $t$-tuples of natural
numbers. Each $t$-tuple in $\HH\cup\{(0,0,\ldots ,0)\}$ corresponds
to a $\GL(V)$ composition factor of $k[\PP]$ and each composition
factor is isomorphic to the $t$-fold twisted tensor product of
exterior powers. These items are discussed in detail in
Section~\ref{GL_structure}.

In Section~\ref{sbf} we shall define a set of polynomials which maps
bijectively to a basis of $k[\PP]$. It is from
this special basis, whose elements will be called {\it symplectic
basis  functions} (SBFs),
 that we will find subsets which give bases of the incidence
submodules $\CC_{r}$ for all $r$. Each SBF is a $t$-fold product
\begin{equation}
\label{factorize}
f=f_0f_1^2\cdots f_{t-1}^{2^{t-1}},
\end{equation}
where each ``digit'' $f_j$ is the function from $V$ to $k$ induced
by a homogeneous square-free polynomial. Homogeneous square-free
polynomials correspond bijectively to elements of the exterior power
of the same degree, and the factorization (\ref{factorize}) is
compatible with the twisted tensor product factorization mentioned
above. Thus each of these functions has a well-defined $\HH$-type.
We also introduce the notions of the {\it class} and {\it level} of
each ``digit'' to further subdivide our set of special functions.
Since the  class and level are defined for each digit they apply
also to elements of exterior powers. In Section~\ref{wedges}, we
show that the filtration of the exterior powers by levels consists
of $k\Sp(V)$-submodules, and compute the dimensions of the
subquotients.

The proof of Theorem~\ref{basis_thm} is then given in a series of
lemmas in Section~\ref{admiss_bases}.  A modified version of the
proof is given in Section~\ref{q=2} to handle the $q=2$ case.
Theorem~\ref{basis_thm} tells us that a basis of $\CC_{r}$ consists
of all those SBFs of certain $\HH$-types and whose digits satisfy
certain conditions on their levels.  Thus, using the dimension
computations of Section~\ref{wedges}, we obtain
Theorem~\ref{rankthm} as a corollary of Theorem~\ref{basis_thm}.

%%%%%%%%%%%%%%%%%%%%%%%%%%%%%%%%%%%%%%%%%%%%%%%%%%%%%%%%%%%%%%%%%%%%%%%%%%%%%%%%

\section{$k[\PP]$ as a permutation module for $\GL(V)$}\label{GL_structure}

In this section, $V$ is a $2m$-dimensional vector space over
$k=\F_q$, $q=2^t$. For the time being, we do not equip $V$ with an
alternating form. So we simply consider $k[\PP]$ as a
$k\GL(V)$-module. Let $k[X_1, X_2,\ldots,X_{2m}]$ denote the
polynomial ring, in $2m$ indeterminates. Since every function on $V$
is given by a polynomial in the $2m$ coordinates $x_i$, the map
$X_i\mapsto x_i$ defines a surjective $k$-algebra homomorphism
$k[X_1, X_2,\ldots,X_{2m}]\rightarrow k[V]$, with kernel generated
by the elements $X_i^q-X_i$. Furthermore, this map is simply the
coordinate description  of the following  canonical map. The
polynomial ring is isomorphic  to the symmetric algebra $S(V^*)$ of
the dual space of $V$; so we have a natural evaluation map
$S(V^*)\rightarrow k[V]$. This canonical description makes it clear
that the map is equivariant with respect to the natural actions of
$\GL(V)$ on these spaces. A basis for $k[V]$ is obtained by taking
monomials in $2m$ coordinates $x_i$ such that the degree in each
variable is at most $q-1$. We will call these the {\it basis
monomials} of $k[V]$. Noting that the functions on $V\setminus\{0\}$
which descend to $\PP$ are precisely those which are invariant under
scalar multiplication by $k^*$, we obtain from the monomial basis of
$k[V]$ a basis of $k[\PP]$,
$$
\{\prod_{i=1}^{2m}x_i^{b_i}\mid 0\leq b_i\leq q-1,
\sum_{i}b_i\equiv 0\;({\rm mod}\;q-1),\;(b_1,\ldots ,b_{2m})\neq
(q-1,\ldots ,q-1)\}.
$$
We refer the elements of the above set as
{\it the basis monomials} of $k[\PP]$.

\subsection{Types and $\HH$-types}\label{HHtypes}
We recall the definitions of two $t$-tuples associated
with each basis monomial. Let
\begin{equation}\label{monomial}
f=\prod_{i=1}^{2m}x_i^{b_i}=\prod_{j=0}^{t-1}\prod_{i=1}^{2m}(x_i^{a_{ij}})^{2^j},
\end{equation}
be a basis monomial of $k[\PP]$, where
$b_i=\sum_{j=0}^{t-1}a_{ij}2^j$ and $0\leq a_{ij}\leq 1$. Let
$\lambda_j=\sum_{i=1}^{2m}a_{ij}$. The $t$-tuple
$\llambda=(\lambda_0,\ldots,\lambda_{t-1})$ is called the {\it type}
of $f$. The set of all types of monomials is denoted by $\Llambda$.

In \cite{bardoe-sin}, there is another $t$-tuple associated with
each basis monomial of $k[\PP]$, which we will call its
\emph{$\HH$-type.} This tuple will lie in the set
$\HH\cup\{(0,0,\ldots 0)\}$, where
\begin{equation*}
\HH=\lbrace \sss=(s_0,s_1,\dots,s_{t-1})\mid \forall j, 1\leq
s_j\leq 2m-1,\ 0\leq 2s_{j+1}-s_j \leq 2m \rbrace.
\end{equation*}
The $\HH$-type ${\bf s}$ of $f$ is uniquely determined by the type
via the equations
\begin{equation*}
\lambda_j=2s_{j+1}-s_j, \quad 0\leq j\leq t-1,
\end{equation*}
where the subscripts are taken modulo $t$. Moreover, these equations
determine a bijection between the set $\Llambda$ of types of basis
monomials of $k[\PP]$ and the set $\HH\cup\{(0,0,\ldots 0)\}$. We
will consider $\HH$ as a partially ordered set under the natural
order induced by the product order on $t$-tuples of natural numbers.

\subsection{Composition factors}
The types, or equivalently the $\HH$-types parametrize the
composition factors of $k[\PP]$ in the following sense. The
$k\GL(V)$-module $k[\PP]$ is multiplicity-free. We can associate to
each $\HH$-type $\sss\in\HH\cup\{(0,0,\ldots 0)\}$
a composition factor, which we shall
denote by $L(\sss)$, such that these simple modules are all
nonisomorphic, with $L((0,0,\ldots ,0))\cong k$. The simple modules
$L(\sss)$, $\sss\in\HH$, occur as subquotients of $k[\PP]$ in the
following way. For each $\sss\in\HH$, we let $Y(\sss)$ be the
subspace spanned by monomials of $\HH$-types in
$\HH_\sss=\{\sss'\in\HH\mid \sss'\leq\sss\}$, and $Y(\sss)$ has a
unique simple quotient, isomorphic to $L(\sss)$.

%peter simplified the next para for char 2
The isomorphism type of the simple module $L(\sss)$ is most easily
described in terms of the corresponding type
$(\lambda_0,\ldots,\lambda_{t-1})\in \Llambda$. Let $S^\lambda$ be
the degree $\lambda$ component in the truncated polynomial ring
$k[X_1, X_2,\ldots,X_{2m}]/(X_i^2; 1\leq i\leq 2m)$. Here $\lambda$
ranges from $0$ to $2m$.  The truncated polynomial ring is
isomorphic to the exterior algebra $\wedge(V^*)$; so the dimension
of $S^\lambda$ is ${2m\choose \lambda}$. The simple module $L(\sss)$
is isomorphic to the twisted tensor
 product
 \begin{equation}\label{tensorproduct}
 S^{\lambda_0}\otimes(S^{\lambda_1})^{(2)}\otimes\cdots\otimes
 (S^{\lambda_{t-1}})^{(2^{t-1})}.
 \end{equation}

\subsection{Submodule structure}\label{substruct}
The reason for considering $\HH$-types is that they allow a simple
description of the submodule structure of the $k\GL(V)$-module
$k[\PP]$. The space $k[\PP]$ has a $k\GL(V)$-decomposition
\begin{equation}\label{splittingofk[P]}
k[\PP]=k\oplus Y_{\PP},
\end{equation}
where $Y_{\PP}$ is the kernel of the map $k[\PP]\rightarrow k$,
$f\mapsto |\PP|^{-1}\sum_{Q\in \PP}f(Q)$. The $k\GL(V)$-module
$Y_{\PP}$ is an indecomposable module whose composition factors are
parametrized by $\HH$. Then \cite[Theorem A]{bardoe-sin} states that
given any $k\GL(V)$-submodule of $Y_{\PP}$, the set of $\HH$-types
of its composition factors is an ideal in the partially ordered set
$\HH$ and that this correspondence is an order isomorphism from the
submodule lattice of $Y_{\PP}$ to the lattice of ideals in $\HH$.

The submodules of $Y_{\PP}$ can also be described in terms of basis
monomials \cite[Theorem B]{bardoe-sin}. Any submodule of $Y_{\PP}$
has a basis consisting of the basis monomials which it contains.
Moreover, the $\HH$-types of these basis monomials are precisely the
$\HH$-types of the composition factors of the submodule.
Furthermore, in any composition series, the images of the monomials
of a fixed $\HH$-type form a basis of the composition factor of that
$\HH$-type.

%%%%%%%%%%%%%%%%%%%%%%%%%%%%%%%%%%%%%%%%%%%%%%%%%%%%%%%%%%%%%%%%%%%%%%%%%%%%%%%%%

\section{symplectic basis functions}
\label{sbf}
We now define a set of square-free homogeneous polynomials of degree $\lambda$
which we will use to construct symplectic basis functions (SBFs).
%An SBF
These square-free homogeneous polynomials are in the indeterminates
$X_1,X_2,\ldots,X_m,Y_m,\ldots,Y_1$, and will have the form
%$$f=f_0{f_1}^2{f_2}^{2^2}\cdots{f_{t-1}}^{2^{t-1}},$$
%where the form of each $f_j$ is
given in the following definition.
\begin{Definition}
Suppose the set of integers $\{1,\ldots,m\}$ is partitioned into
five disjoint sets, $R,\ R',\ S,\ T,$ and $U$.  We require
$|R|=|R'|=\rho$, and take the cardinalities of the other sets to be
$ \sigma,\ \tau,$ and $\upsilon,$ respectively, such that
$2\rho+2\sigma+\tau=\lambda$. We further impose a pairing between
$R=\{r_1,\ldots,r_\rho\}$ and $R'=\{r'_1,\ldots,r'_\rho\}$, so that
$r_i$ is paired with $r'_i,\ 1\le i\le \rho$.  We write
$W_i=X_iY_i,\ 1\le i\le m$, and write $Z_i$ to denote either $X_i$
or else $Y_i$. Then we shall say that a square-free homogeneous
polynomial $F$ belongs to the \emph{class}
$(\rho,\sigma,\tau,\upsilon)$ if it can be written as
$$F=\prod_{i=1}^{\rho}(W_{r_i}+W_{r'_i}) \prod_{i\in S}W_i \prod_{i\in T}
Z_i.$$

We denote by $P^\lambda_\ell$ the set of polynomials of those classes
$(\rho,\sigma,\tau,\upsilon)$ with
$2\rho+2\sigma+\tau=\lambda$ and $\sigma\le\ell$.  Then $\sigma$ will be called the
\emph{level} of $F$, and $\ell$ will be called the \emph{level} of $P^\lambda_\ell$.
\end{Definition}

We now construct a basis for $k[\mathcal P]$. We begin with a basis
for the space of square-free homogeneous polynomials of degree
$\lambda$ in $X_1,Y_1,\ldots,X_m,Y_m$.  Note that the linear span of
$P^\lambda_{\ell-1}$ is a subspace of the span of $P^\lambda_\ell$.

\begin{Definition}
Let $\lambda$ be a fixed integer, $0\le\lambda\le2m$.  Choose $B^\lambda_0$ to be a  linearly independent subset of $P^\lambda_0$ of maximum size.  For $1\le\ell\le\lambda/2$, suppose that $B^\lambda_{\ell-1}$ has been chosen.  Then we choose
 ${P'}^\lambda_\ell\subset P^\lambda_\ell$ such that
$B^\lambda_\ell=B^\lambda_{\ell-1}\cup {P'}^\lambda_\ell$ forms a basis for the span of $P^\lambda_\ell$.  Finally, let
$$\mathcal B=\bigcup_{0\le\lambda\le2m}B^\lambda_{\lfloor\lambda/2\rfloor}.$$
\end{Definition}

Note that $B^\lambda_{\lfloor\lambda/2\rfloor}$ is a basis for the
space of all square-free homogeneous polynomials of degree
$\lambda$. Thus $\mathcal B$  forms a basis of the space of
square-free polynomials in $X_1,Y_1,\ldots,X_m,Y_m$. Polynomials of
the form $\mathcal F=F_0{F_1}^2\cdots{F_{t-1}}^{2^{t-1}}$, where
$F_j\in \mathcal B,\ 0\le j\le t-1$ form a basis for polynomials in
these indeterminates, where the exponent of each indeterminate is at
most $q-1$.

The evaluation map $\phi:
k[X_1,\ldots,X_m,Y_m,\ldots,Y_1]\rightarrow k[V]$, sending $X_i$ to
$x_i$ and $Y_i$ to $y_i$, $1\leq i\leq m$, restricts to a bijection
from the space of square-free homogeneous polynomials in
$X_1,\ldots,X_m,Y_m,\ldots,Y_1$ to those in
$x_1,\ldots,x_m,y_m,\ldots,y_1$. If $F$ is a square-free polynomial
belonging to class $(\rho,\sigma,\tau,\upsilon)$, then we shall
refer to the function $\phi(F)$ as belonging to that class also.

\begin{Definition}
Let $(s_0,\ldots,s_{t-1})\in\HH$ and let $f=f_0{f_1}^2\cdots
{f_{t-1}}^{2^{t-1}}\in k[\PP]$ such that the degree of $f_j$ is
$\lambda_j=2s_{j+1}-s_j$ {\rm(}subscripts {\rm modulo $t$),} and
$f_j\in \phi(B^{\lambda_j}_{\lfloor \frac {\lambda_j} {2}
\rfloor})$, for each $j,\ 0\le j\le t-1$.  Then we shall say that
$f$ is a \emph{symplectic basis function} ({SBF}), of the
specified $\HH$-type.
\end{Definition}

\subsection{Linear substitutions in factorizable functions}\label{subs}
We will frequently be dealing with functions, such as basis monomials
or SBFs, which can be factorized as a product
\begin{equation}\label{digits}
f=f_0{f_1}^2\cdots{f_{t-1}}^{2^{t-1}}\in k[\PP]
\end{equation}
of functions $f_j\in k[V]$ which are images under $\phi$ of
homogeneous polynomials. We wish to think of  $f_j$ as the \emph{$j^{\mathrm{th}}$
digit of $f$}.  However the representation (\ref{digits})
is not unique in general; $f=x_1^2$ can be factorized by taking
$f_0=x_1^2$ or by taking $f_1=x_1$.
Thus when we refer to the digit of such a function,
it must be with a particular factorization in mind.
If we impose the additional requirement that the $f_j$ in
(\ref{digits}) be square-free, then the digits are
determined up to scalars and such a function $f$ has a well-defined $\HH$-type,
as in the case of basis monomials and SBFs.
The more general notion of digit is needed
to discuss transformations of functions.
Let $b=b_0{b_1}^2\cdots{b_{t-1}}^{2^{t-1}}$ be a function
which has a well-defined $\HH$-type.
An element $g\in\GL(V)$ acts on $k[V]$ as a linear substitution,
and we have
\begin{equation}\label{linsubs}
gb=(gb_0){(gb_1)}^2\cdots{(gb_{t-1})}^{2^{t-1}}.
\end{equation}
The digits $gb_j$ in (\ref{linsubs}) are the homogeneous polynomials
obtained from the $b_j$ through linear substitution, and may
contain monomials with square factors.
Like any function, $gb$ can be rewritten as a linear combination of
factorizable functions with square-free digits, such as
basis monomials. We want to take a closer look at this
rewriting process, keeping track of the relationship
between the $\HH$-type of $b$  and the $\HH$-types of the
terms in the possible rewritten forms of $gb$.
The rewriting as a combination of basis monomials is done as follows.
The product (\ref{linsubs}) is first distributed into
a sum of products of monomials, each monomial having the factorized form
(\ref{digits}), possibly with square factors in the digits.
Then if the first digit has a factor $x_i^2$, the factor $x_i^2$
is replaced by 1 and the second digit is multiplied
by $x_i$. We will call this a \emph{carry} from the first digit
to the second. Note that this process causes the degree of the first
digit to decrease by $2$ while that of the second increases by one.
We perform all possible carries from the first digit to the second, leaving the first digit square-free. We then repeat the process, performing carries from the second digit to the third, and so on. Carries from the last digit
go to the first digit, and after that a second round of carries
from the first to second digits, second to third, etc. may occur.
This process terminates because each time we have a carry
from the last digit to the first, the total degree decreases by $2^t-1$.
To refine this idea we define, for $0\leq e\leq t-1$, the $e^{\mathrm{th}}$ \emph{twisted degree} of (\ref{digits}) by
$$
\deg_e(f)=\sum_{j=0}^{t-1}2^{[j-e]}\deg(f_j),
$$
where $[j-e]$ means the remainder modulo $t$.
The effect of a carry from the $(j-1)^{\mathrm{th}}$ digit to the $j^{\mathrm{th}}$ digit
is to lower the $j^{\mathrm{th}}$ twisted degree by one, leaving other twisted
degrees unchanged. For those functions $f$ which have a well-defined $\HH$-type
the relation between the $\HH$-type $(s_0,\ldots,s_{t-1})$
and the twisted degrees is simple: $(2^t-1)s_e=\deg_e(f)$.
We conclude that in the process of rewriting $gb$,
whenever a carry is performed on a monomial
the basis monomial which results will have a strictly lower $\HH$-type
than our original function $b$.

%%%%%%%%%%%%%%%%%%%%%%%%%%%%%%%%%%%%%%%%%%%%%%%%%%%%%%%%%%%%%%%%%%%%%%%%%%%%%%%%%

\section{Submodules of the exterior powers}\label{wedges}

The ring $k[X_1,\ldots,X_m,Y_1,\ldots, Y_m]/(X_i^2,Y_i^2)_{i=1}^m$
can be viewed as the exterior algebra $\wedge(V^*)$, with
$X_i$ and $Y_i$ corresponding to the basis elements $x_i$ and $y_i$ of
$V^*$ respectively.
Therefore, we can think of elements of  $\wedge(V^*)$
as square-free polynomials in the $X_i$ and $Y_i$ and
identify $\wedge^\lambda(V^*)$ with $S^\lambda$, for $0\leq \lambda\leq2m$.

For each $\lambda$, we define
$S^\lambda_\ell$ to be the $k$-span of $P^\lambda_\ell$, which is also the $k$-span of $B^\lambda_\ell$.  Then the levels of functions define
a filtration of subspaces
$$
0\subset S^\lambda_0\subseteq\cdots \subseteq
S^\lambda_{\lfloor\frac{\lambda}2\rfloor} =S^\lambda.
$$
Note that if $\lambda\geq m$ then $S^\lambda_\ell=0$ for $\ell<\lambda-m$.

\begin{Lemma}\label{submodule} Each $S^\lambda_\ell$ is a $k\Sp(V)$-submodule of $S^\lambda$.
\end{Lemma}
%\begin{proof}Insert David's proof here.
%\end{proof}
We delay the proof of Lemma~\ref{submodule} until after the similar proof of Lemma~\ref{submodule lemma}.

\begin{Lemma}\label{Weyl} Assume $\lambda\leq m$.
Then $S^\lambda_0$ has dimension $\binom{2m}{\lambda}
-\binom{2m}{\lambda-2}$. A basis of $S^\lambda_0$ consists of all
elements of $P^\lambda_0$ of the form
$\prod_{i=1}^\rho(W_{r_i}+W_{r'_i})\prod_{t\in T}Z_t$, where
$R=\{r_1,r_2,\ldots ,r_{\rho}\}$, $R'=\{r'_1,r'_2,\ldots ,r'_\rho\}$
and $T$ are disjoint subsets of $\{1,\ldots, m\}$ such that
$2\abs{R}+\abs{T}=\lambda$ and the following conditions hold.
\begin{enumerate}
\item $r_1<r_2<\cdots <r_\rho$ and $r'_1<r'_2<\cdots <r'_\rho$.
\item $r_i$ is the smallest element of $\{1,\ldots, m\}\setminus T$
which is not in the set $\{r_j, r'_j \mid\, j<i\}$.
\end{enumerate}
\end{Lemma}
\begin{proof} A proof can be found in \cite{brouwer}, Theorem 1.1;
an earlier proof appears in the examples of \cite{jantzen_thesis}
p.~39. (We only need the statements in the characteristic 2 case,
but these references treat all characteristics.)
\end{proof}

%added ``interested reader'' to indicate not required.
\begin{Remark} The module  $S^\lambda_0$ or, more precisely, their
scalar extensions to an algebraic closure $\overline k$
are examples of {\it Weyl} modules for the algebraic group
$\Sp(\overline k\otimes_kV)$ and it is in this context that
they were studied in \cite{jantzen_thesis} and \cite{brouwer}.
We refer the interested reader to \cite{jantzen_book} for the definitions
and properties of this important class of modules.
\end{Remark}

\begin{Lemma}\label{wfilt} Assume $\lambda\leq m$.
\begin{enumerate}[\rm 1.]
\item We have an isomorphism
\begin{equation} S^\lambda_\ell/S^\lambda_{\ell-1}\cong S^{\lambda-2\ell}_0.
\end{equation}
\item The dimension of $S^\lambda_\ell$ is $\binom{2m}{\lambda}-\binom{2m}{\lambda-2\ell-2}$.
\end{enumerate}
\end{Lemma}
\begin{proof} We shall define a linear map
$$
\alpha: S^{\lambda-2\ell}_0 \rightarrow S^\lambda_\ell.
$$
Let $f$ be a basis element of $S^{\lambda-2\ell}_0$ given by
Lemma~\ref{Weyl}, with index sets $(R,R',T)$.
Set $M=\{1,\ldots,m\}$. Then the number of indices
involved in $f$ is $2\rho+\tau=\lambda-2\ell$, and the number of unused indices is $(m-\lambda)+2\ell\ge 2\ell$.
%Since
%$\lambda\leq m$ there
%are at least $2\ell$ indices from $\{1,\ldots,m\}\setminus T$
%which do not appear in $f$.
We choose an arbitrary $\ell$-subset $S$
of these unused indices and define
$$
\alpha(f)=f\prod_{s\in S}W_s.
$$
Composing with the natural projection gives a map
$$\overline\alpha:\quad
S^{\lambda-2\ell}_0 \to S^\lambda_\ell/S^\lambda_{\ell-1}.$$

We claim that the map $\overline\alpha$ is surjective.

Now $S^\lambda_\ell$ is spanned by polynomials of the form

\begin{equation}\label{spanpolys}
\prod_{r_i\in R,r'_i\in R'}(W_{r_i}+W_{r'_i})\prod_{s_i\in S}W_{S_i}\prod_{t_i\in T}Z_{t_i},
\end{equation}
where $R$, $R'$, $S$ and $T$ are disjoint subsets of
$M=\{1,\ldots,m\}$, with $\abs{R}=\abs{R'}$, and
$2\abs{R}+2\abs{S}+\abs{T}=\lambda$, and $\abs{S}\leq\ell$. The
classes of these elements such that $\sigma=\ell$ span the quotient
space $S^\lambda_\ell/S^\lambda_{\ell-1}$. We consider what
restrictions may be placed on the form of such a polynomial and
still have a spanning set. Of course we  may assume that the $r_i$
form an increasing sequence and that $r_i<r'_i$ for all $i$.
Furthermore, the equation
\begin{equation}\label{subs1}
(W_1+W_4)(W_2+W_3)=(W_1+W_3)(W_2+W_4)+(W_1+W_2)(W_3+W_4)
\end{equation}
implies that we can assume that the $r'_i$ also form an increasing
sequence.
Next observe that
\begin{equation}\label{subs2}
(W_2+W_3)=(W_1+W_2)+(W_1+W_3)
\end{equation}
and
\begin{equation}\label{subs3}
(W_2+W_3)W_1=(W_1+W_3)W_2 +(W_1+W_2)W_3.
\end{equation}

Equations (\ref{subs2}) and (\ref{subs3}) respectively allow us to
assume in addition that no element of $U$ or  $S$ precedes any
element of $R$, or in other words that for each $i$, $r_i$ is the
smallest index in $(\{1,\dots,m\}\setminus
T)\setminus\{r_j,r_j'\mid\, j<i\}$. Thus the  sets $R$ and $R'$ can
be assumed to satisfy the same conditions as the corresponding sets
for the basis elements of $S^{\lambda-2\ell}_0$ stated in
Lemma~\ref{Weyl}. We have proved that there is a spanning set of
$S^\lambda_\ell$ such that for each element for which $\sigma=\ell$,
the subsets $R$ and $R'$ and $T$ are exactly the same as those of
some $\alpha(f)$, where $f$ is a basis element of
$S_0^{\lambda-2\ell}$. Finally, we note that two elements of
$P^\lambda_\ell$ with the same sets $R$, $R'$ and $T$ are equal
modulo $S^\lambda_{\ell-1}$. To see this, suppose the index sets of
the two elements are $(R,R',S,T,U)$ and $(R,R',S^*,T,U^*)$ and
assume that the symmetric difference of $S$ and $S^*$ is
$\{s,s^*\}$. Then the sum of the two polynomials has index sets
$(R\cup\{s\},R'\cup\{s^*\},S\cap S^*, T,(U\cup
U^*)\setminus\{s,s^*\})$; so it belongs to $S^\lambda_{\ell-1}$.
Therefore the map $\overline\alpha$ is surjective (and,
incidentally, independent of the choice of $\ell$-subset $S$ in the
definition of $\alpha$).

Thus, since Lemma~\ref{Weyl} gives the dimension of $S^{\lambda-2\ell}_0$,
we have
\begin{equation}
\begin{aligned}
\binom{2m}{\lambda}=\dim S^\lambda&=\sum_{\ell=0}^{\lfloor\frac\lambda 2\rfloor}\dim (S^\lambda_\ell/S^\lambda_{\ell-1})\\
&\leq \sum_{\ell=0}^{\lfloor\frac\lambda 2\rfloor}\dim S^{\lambda-2\ell}_0\\
&=\sum_{\ell=0}^{\lfloor\frac\lambda 2\rfloor}\binom{2m}{\lambda-2\ell}-
\binom{2m}{\lambda-2\ell-2}\\
&=\binom{2m}{\lambda},
\end{aligned}
\end{equation}
so equality holds throughout. Both parts of the lemma now follow.
\end{proof}

The case where  $\lambda\geq m$ will be treated by using various
dualities, which we now proceed to discuss. We do not assume
 $\lambda\geq m$ yet.
We start with the $k\GL(V)$-isomorphism
$\wedge^\lambda(V^*)\cong(\wedge^\lambda(V))^*$.
Let $\lambda^*=2m-\lambda$. Then there is a natural pairing
$$
\wedge^\lambda(V)\times\wedge^{\lambda^*}(V)\rightarrow \wedge^{2m}(V)=k
$$
given by exterior multiplication, which defines a $k\GL(V)$-isomorphism
$(\wedge^{\lambda}(V))^*\cong \wedge^{\lambda^*}(V)$.

Finally, the $k\Sp(V)$-isomorphism $V\cong V^*$ given by $v\mapsto b(v,\cdot)$
induces a $k\Sp(V)$-isomorphism
$\wedge^{\lambda^*}(V)\cong\wedge^{\lambda^*}(V^*)$.
Combining these isomorphisms yields a $k\Sp(V)$-module isomorphism
$$
S^\lambda\rightarrow S^{\lambda^*}, \quad
X_IY_J\mapsto X_{M\setminus J}Y_{M\setminus I},
$$
for each $\lambda$.
Putting these isomorphisms together for all $\lambda$ yields
a $k\Sp(V)$-automorphism $\delta$ of $\wedge(V^*)$ of order 2.
A straightforward calculation shows that a basis polynomial
of degree $\lambda^*$ and class $(\rho,\sigma,\tau,\upsilon)$
is mapped under $\delta$ to a basis polynomial
of degree $\lambda$ and class $(\rho,\upsilon,\tau,\sigma)$.
(Actually, the map $\delta$ is  very simple; a basis polynomial of
$S^{\lambda^*}$ with index sets $(R,R',S,T,U)$ is mapped to  basis polynomial
of $S^\lambda$ with index sets $(R,R',U,T,S)$; $\delta$ just swaps $U$
and $S$.)
Now assume $\lambda\geq m$.
Since $m=2\rho+\sigma+\tau+\upsilon$ and $\lambda^*=2\rho+2\sigma+\tau$,
we have
$$
\upsilon=m-\lambda^*+\sigma=(\lambda-m)+\sigma,
$$
which shows that
$$
\delta(S^{\lambda^*}_{\ell^*})=S^\lambda_{(\lambda-m)+\ell^*},
$$
for $0\leq\ell^*\leq\lfloor\frac{\lambda^*}{2}\rfloor$.

We have already computed the left hand side, so if we set
$\ell=(\lambda-m)+\ell^*$ then
$$
\dim S^\lambda_\ell=
\binom{2m}{\lambda^*}-\binom{2m}{\lambda^*-2\ell^*-2}=
\binom{2m}{\lambda}-\binom{2m}{\lambda-2\ell-2}.
$$

We have proved:
\begin{Theorem}\label{dimensions} Let $0\leq\lambda\leq 2m$
and $\max\{0,\lambda-m\}\leq\ell\leq \lfloor\frac{\lambda}{2}\rfloor$.
Then $S^\lambda_\ell$ has dimension
$$
\binom{2m}{\lambda}-\binom{2m}{\lambda-2\ell-2}.
$$
\end{Theorem}

\begin{Remark} The isomorphism in Lemma~\ref{wfilt} is in fact one
of $k\Sp(V)$-modules, and after extending the field  to $\overline k$
it is an isomorphism of rational $\Sp(V_{\overline k})$-modules. In
fact, it can be shown that the filtration given by the submodules
$\overline k\otimes_kS^\lambda_\ell$ is the filtration of
$\wedge^\lambda(V^*_{\overline k})$ by Weyl modules, dual to the
\emph{good filtration} described in \cite{donkin}, Appendix A, p.71
(for all characteristics).
\end{Remark}

%%%%%%%%%%%%%%%%%%%%%%%%%%%%%%%%%%%%%%%%%%%%%%%%%%%%%%%%%%%%%%%%%%%%%%%%%%%%%%%%%

\section{Bases for incidence modules}
\label{admiss_bases}

In this section we assume that $q>2$.  In the next section we show
that Theorem~\ref{basis_thm} below also holds for $q=2$, with a
slightly different proof.
\begin{Definition}
Let $1\leq r\leq 2m-1$ and let  $f=f_0{f_1}^2\cdots
{f_{t-1}}^{2^{t-1}}\in k[\PP]$ with $f_j$ homogeneous
and square-free for all $j$. If the $\HH$-type of $f$ is
$(s_0,s_1,\ldots ,s_{t-1})$, let
\begin{equation}\label{r-admit}
\ell_j=(m-r)\delta(r\leq m) + (2m-r-s_j).
\end{equation}
The function $f$ is called $r$-\emph{admissible} if $f$ is a
constant function, or if
\begin{enumerate}[\rm (a)]
\item for each $j$, $0\leq j\leq t-1$, $f_j\in \phi(P^{\lambda_j}_{\ell_j})$,
and
\item the $\HH$-type  of $f$ is less than or equal to
$(2m-r,2m-r,\ldots ,2m-r)$.
\end{enumerate}
In particular, if $f$ is constant or if
$f_j\in\phi(B^{\lambda_j}_{\ell_j})$ for all $j$,
$0\leq j\leq t-1$, we call $f$ an $r$-\emph{admissible SBF.}
\end{Definition}
Note that by construction, $r$-admissible functions are in the linear span of $r$-admissible SBFs.

\begin{Theorem}\label{basis_thm}
Assume $q>2$. The $r$-admissible symplectic basis
functions form a basis for the
$k\Sp(V)$-submodule $\CC_r$ of $k[\mathcal P]$.
\end{Theorem}

The rest of this section is devoted to proving this theorem.

Let $\MM_r$ denote the linear span of the
$r$-admissible  SBFs, for $1\le r\le 2m-1$.

Let $\overline P^\lambda_\ell=\phi(P^{\lambda}_{\ell})$.

 We consider the following group elements, which together  generate $\Sp(V)$.
Let $\alpha\in k^\times$, and let $\pi$ be a permutation of $\{1,\ldots,m\}$.
\begin{gather}
x_1\mapsto \alpha x_1,\quad y_1\mapsto \alpha^{-1}y_1\label{gen1}\\
x_k\mapsto x_{\pi(k)},\;y_k\mapsto y_{\pi(k)},\quad1\le k\le m\label{gen2}\\
x_1\mapsto y_1,\quad y_1\mapsto x_1\label{gen3}\\
g_1(\alpha): x_1\mapsto x_1+\alpha y_1,\quad y_1\mapsto y_1 \label{transvection}\\
g_2(\alpha): x_1\mapsto x_1+\alpha x_2,\ y_1\mapsto y_1,\ x_2\mapsto x_2,\
y_2\mapsto \alpha y_1+y_2.\label{symptransvection}
\end{gather}

\begin{Lemma}
\label{submodule lemma}
$\MM_r$ is a $k\Sp(V)$-submodule of $k[\mathcal P]$.
\end{Lemma}
\begin{proof}
Since $\Sp(V)$ acts as the identity on a constant function, we only
need to consider the action of $\Sp(V)$ on an SBF of $\HH$-type
$(s_0,\ldots,s_{t-1})$. For each fixed $\lambda,\ 0\le\lambda\le
2m,$ and $\ell,\ 0\le\ell\le\lambda/2$. Since
$m=2\rho+\sigma+\tau+\upsilon$, we note that if $\ell<\lambda-m$,
then $\overline P^\lambda_\ell$ is empty, while if $\ell>\lambda/2$,
then $\overline P^{\lambda}_{\ell}=\overline
P^{\lambda}_{\lfloor\lambda/2\rfloor}$.

We examine one by one the actions of generators of $\Sp(V)$ on
elements of our spanning set of functions.
First, since our set is defined symmetrically with respect to the
subscripts, it is clear that elements (\ref{gen2}) of $\Sp(V)$ which
permute the subscripts of the standard basis of $V^*$ also permute
the members of $\overline P^\lambda_\ell$.  It is also clear that
interchanging $x_i$ and $y_i,\ 1\le i\le m$ (\ref{gen3}), also
interchanges some members of $\overline P^\lambda_\ell$.  The
diagonal group elements (\ref{gen1}) multiply each $z_i$ by some
constant and leave each $w_i$ unchanged.  Thus they only multiply
members of $\overline P^\lambda_\ell$ by finite field elements.  It
remains to consider the actions of the
group elements (\ref{transvection}) and (\ref{symptransvection}).
%\begin{eqnarray}
%&x_1\mapsto x_1+\alpha y_1,\ y_1\mapsto y_1 &\label{transvection}\\
%&\mathrm{and}&\nonumber\\
%&x_1\mapsto x_1+\alpha x_2,\ y_1\mapsto y_1,\ x_2\mapsto x_2,\
%y_2\mapsto \alpha y_1+y_2.&\label{symptransvection}
%\end{eqnarray}

%Let $g_1(\alpha)$ denote the element represented in (\ref{transvection}), and %let $g_2(\alpha)$ denote the element represented in (\ref{symptransvection}).
Let
$f=f_0{f_1}^2\cdots {f_{t-1}}^{2^{t-1}}$ be an $r$-admissible SBF, and let
$f_j\in\overline P^{\lambda_j}_{\ell_j}$.
First we will examine the action of $g_1(\alpha)$ or $g_2(\alpha)$ on the $j^\mathrm{th}$
digit $f_j$ assuming there is no carry from the $(j-1)^\mathrm{th}$ digit
to the $j^\mathrm{th}$ digit.  Then we consider the effect of such a carry.

In calculating $g_1(\alpha)f_j$, it is sufficient to write just the part of $f_j$ involving $x_1$ or $y_1$, except when $1\in R$.  Whenever we get a square term, $y_1^2$ or $x_2^2$, we replace it by $1$, reflecting the carry to $f_{j+1}$.  Thus we compute $g_1(\alpha)f_j$ modulo $(y_1^2-1)$. Recall $w_1=x_1y_1$.  We have
\begin{equation}\label{sigmadec}
g_1(\alpha)w_1=w_1+\alpha y_1^2\equiv w_1+\alpha
\end{equation}
\begin{equation}
\begin{aligned}
g_1(\alpha)(w_1+w_2)&=(w_1+w_2)+\alpha y_1^2\equiv (w_1+w_2)+\alpha\\
g_1(\alpha)x_1&=x_1+\alpha y_1\\
g_1(\alpha)y_1&=y_1\\
g_1(\alpha)(1)&=1.
\end{aligned}
\end{equation}
In each case we get a linear combination of elements of
$\overline P^{\lambda_j}_{\ell_j}$; that is, $\sigma_j$ did decrease in one term of (\ref{sigmadec}), but it never increased.
Similarly, we present the following calculations, exhaustively showing
the action of $g_2(\alpha)$ on the portion of $f_j$ involving
%factor $f$, which does not involve
$x_1,\ y_1,\ x_2,$ and $y_2$.
% nor $x_4$ nor $y_4$ in (\ref{eq3}),
%nor $x_3$ nor $y_3$ in (\ref{eq3})--(\ref{eq7}).
The calculation is modulo $\left((y_1^2-1),(x_2^2-1)\right)$.
\begin{eqnarray}
g_2(\alpha)w_1w_2&=&(x_1+\alpha x_2)y_1x_2(y_2+\alpha y_1)\nonumber\\
&=&w_1w_2+\alpha y_1y_2x_2^2+\alpha x_1x_2y_1^2+\alpha^2y_1^2x_2^2\nonumber\\
&\equiv&w_1w_2+\alpha y_1y_2+\alpha x_1x_2+\alpha^2\label{eq1}
\\%\end{eqnarray}\begin{eqnarray}
g_2(\alpha)(w_1+w_2)&=& (w_1+w_2)+\alpha x_2y_1+\alpha x_2y_1 \nonumber \\
  &=&(w_1+w_2)
\\%\end{eqnarray}\begin{eqnarray}
g_2(\alpha)(w_1+w_3)(w_2+w_4)&=& (w_1+w_3)(w_2+w_4)+\alpha x_2y_1w_4+\nonumber\\
&&\alpha w_3x_2y_1+\alpha y_1y_2x_2^2+\nonumber\\&&\alpha x_1x_2y_1^2+
\alpha^2y_1^2x_2^2\nonumber\\
&\equiv&(w_1+w_3)(w_2+w_4)+\alpha y_1y_2+\nonumber\\
&&\alpha x_1x_2+\alpha x_2y_1(w_3+w_4)+
\alpha^2
  \label{eq3}
\end{eqnarray}\begin{eqnarray}
g_2(\alpha)(w_1+w_3)w_2&\equiv&(w_1+w_3)w_2+ \alpha w_3x_2y_1+\nonumber\\
&&\alpha y_1y_2+\alpha x_1x_2+\alpha^2\label{eq4}
\end{eqnarray}\begin{eqnarray}
g_2(\alpha)(w_1+w_3)&=&(w_1+w_3)+\alpha x_2y_1   \\
g_2(\alpha)(w_1+w_3)x_2&\equiv& (w_1+w_3)x_2 +\alpha y_1\label{eq6}
\end{eqnarray}\begin{eqnarray}
g_2(\alpha)(w_1+w_3)y_2&\equiv& (w_1+w_3)y_2 +\alpha y_1x_2y_2+\nonumber\\
&&\alpha x_1+\alpha y_1w_3+\alpha^2x_2 \nonumber\\
&=&(w_1+w_3)y_2 +\alpha y_1(w_2+w_3)+\nonumber\\
&&\alpha x_1+\alpha^2x_2 \label{eq7}
\end{eqnarray}\begin{eqnarray}
g_2(\alpha)w_1&=& w_1+\alpha x_2y_1   \\
g_2(\alpha)w_1x_2&\equiv& w_1x_2+\alpha y_1\label{eq9}  \\
g_2(\alpha)w_1y_2&=& w_1y_2 +\alpha y_1w_2 +\alpha x_1+\alpha^2 x_2\label{eq10}  \\
g_2(\alpha)x_1x_2&\equiv& x_1x_2 +\alpha\label{eq11}  \\
g_2(\alpha)x_1y_2&=& x_1y_2 +\alpha (w_1+w_2)+\alpha^2y_1x_2\label{eq12}  \\
g_2(\alpha)x_2y_1&=&x_2y_1\label{eq13}.
\end{eqnarray}
We see that in no term does $\sigma_j$ increase. Notice that $s_j$
never increases under the action of $\Sp(V)$, but decreases in the
case of a carry from the $(j-1)^\mathrm{th}$ to the $j^\mathrm{th}$
digit.  Thus, $\ell_j$ never decreases, where
$\ell_j=(m-r)\delta(r\le m)+2m-r-s_j$. Therefore, not considering
any carries into a digit, $\overline{P}^{\lambda_j}_{\ell_j}$ is
stable under the actions of $g_1(\alpha)$ and $g_2(\alpha)$.

We also take into account that in some terms, $y_1^{2^j}$ might be
carried from $f_{j-1}$ to $f_j$, and possibly also $x_2^{2^j}$ in
the case of $g_2(\alpha)$.  The carry is caused by the action of
$g_1(\alpha)$ or $g_2(\alpha)$ on $f_{j-1}$.  In these cases,
$\sigma_j$ might increase.  We consider a carry of the exponent
$b_1$ of $y_1$, as the case for $x_2$ is similar.  The carry reduces
$\lambda_{j-1}$ by $2$ and increases $\lambda_j$ by $1$;  therefore
$s_j$ is decreased by $1$, and $\ell_j$ is increased by $1$. The
following illustrates that $\sigma_j$ increases by at most $1$ when
a factor of $y_1$ is carried from $f_{j-1}$ to $f_j$.
\begin{eqnarray*}
x_1&\to&w_1\\
y_1&\to&1\\
1&\to&y_1\\
w_1&\to&x_1\\
w_1+w_3&\to&x_1+y_1w_3
\end{eqnarray*}
In summary, $\ell_j$ increased by $1$ in each case, while $\sigma_j$
increased, only in the first case and the second term of the last
case. Thus $\sigma_j\le\ell_j$ is still true.
%\qed
\end{proof}

\noindent{\bf Proof of Lemma~\ref{submodule}.} The calculations here
are similar to those in the proof of Lemma~\ref{submodule lemma}. In
fact the calculations are easier since they are done modulo $X_i^2$
and $Y_i^2$, which means that there are no carries to
consider. We will not repeat the detailed computations here.\qed
\vspace{0.1in}

\begin{Lemma}\label{Cinside} $\CC_r\subseteq \MM_r$.
\end{Lemma}
\begin{proof}
 Since
$\CC_r$ is generated as a $k\Sp(V)$-submodule of $k[\mathcal P]$ by
the characteristic function of any element of $\Ii_r$, we pick one
such element and write its characteristic function as
\begin{eqnarray}%\label{charfunction}
\nonumber (1-{x_1}^{q-1})\cdots(1-{x_m}^{q-1})
(1-{y_1}^{q-1})\cdots(1-{y_{m-r}}^{q-1}) &\mathrm{if}&r<m\\
(1-{x_1}^{q-1})\cdots(1-{x_{2m-r}}^{q-1})\hspace{1 in} &\mathrm{if}&r\ge
m.\label{charfunctioneq}
\end{eqnarray}
Each monomial in the expansion of this function is in $\MM_r$.
Indeed, take any nonconstant monomial $f$ in the expansion of either
characteristic function above, and assume that its $\HH$-type is
$(s_0,s_1,\ldots,s_{t-1})$.  Then
$(s_0,s_1\ldots,s_{t-1})\le(2m-r,2m-r,\ldots,2m-r)$.  Also, for each
digit of $f$, $\sigma_j=0$ if $r\ge m$, and $\sigma_j\le m-r$ if
$r\le m$.  Since $2m-r-s_j$ is nonnegative, it is clear that
$f_j\in\overline P^{\lambda_j}_{\ell_j}$, where
$\ell_j={(m-r)\delta(r\le m)+(2m-r-s_j)}$.  Since also the
$\HH$-type is at most $(2m-r,\ldots,2m-r)$, the monomial is
$r$-admissible. Therefore $\CC_r\subseteq \MM_r$.
\end{proof}

In order to prove Theorem~\ref{basis_thm}, we will show that every
$r$-admissible SBF can be obtained from an element of $\CC_r$ by
applying operators from the group ring $k\Sp(V)$.

The next two lemmas provide us with operators from the group
ring which allow us to control the shape of a polynomial.
The utility of these operators lies in the fact that they modify
factorizable functions in only one digit.

\begin{Lemma}
\label{shifts}
For $0\leq j\leq t-1$, let
\begin{equation}\label{shiftoper}
g(j) = \sum_{\mu\in k^{\times}} \mu^{2^j}g_1(\mu^{-1})\in k[{\rm Sp}(V)].
\end{equation}
Given any basis monomial $f=x_1^{a_1}y_1^{b_1}\cdots
x_m^{a_m}y_m^{b_m}$ of $k[V]$, we have
$$g(j) f= \left\{\begin{array}{ll}
    0, & \textrm{\emph{if the $j^\mathrm{th}$ digit of $a_1$ is 0,}} \\
    x_1^{a_1- 2^j}y_1^{b_1+  2^j}x_2^{a_2}y_2^{b_2}\cdots x_m^{a_m}y_m^{b_m}, & \textrm{\emph{if the
$j^\mathrm{th}$ digit of $a_1$ is 1.}}
            \end{array}\right. $$
\end{Lemma}

\begin{proof} We first prove the lemma for the $j=0$ case. If $a_1=0$, then
clearly we have $g (0)f=(\sum_{\mu\in k^{\times}}\mu)f=0$. So we
assume that $a_1 > 0$.
\begin{eqnarray*}
g(0) f &= & \sum_{\mu\in k^{\times}} \mu (x_1 + \mu^{-1}
y_1)^{a_1}y_1^{b_1}x_2^{a_2}y_2^{b_2}\cdots
x_m^{a_m}y_m^{b_m}\\
&= &\sum_{\mu\in k^{\times}} \mu \left(x_1^{a_1} + {a_1 \choose 1}\mu^{-1}x_1^{a_1
-1}y_1+{a_1\choose 2}\mu^{-2} x_1^{a_1 -2}y_1^2 +\cdots\right)\\
&\cdot& y_1^{b_1}x_2^{a_2}y_2^{b_2}\cdots
x_m^{a_m}y_m^{b_m}\\
&= & a_1x_1^{a_1-1}y_1^{b_1
+1}x_2^{a_2}y_2^{b_2}\cdots x_m^{a_m}y_m^{b_m}.
\end{eqnarray*}
Therefore the lemma is proved in the case where $j=0$. The general
case follows from the $j=0$ case by applying the Frobenius
automorphism.
\end{proof}

\begin{Lemma}\label{project}
Let $1\le i\le m$, let $0\le j\le t-1$, and let
$f={x_1}^{a_1}{y_1}^{b_1}\ldots {x_m}^{a_m}{y_m}^{b_m}$ be a basis monomial of
$k[V]$.  Then there exist projectors
$p_{i,j}^{(1)},\ p_{i,j}^{(2)}$, and $p_{i,j}^{(3)}$ such that
\begin{eqnarray*}
p_{i,j}^{(1)}(f)&=&\left\{
\begin{array}{ll}
f,&\mathrm{if\ the\ } j^\mathrm{th}\ \mathrm{digits\ of}\ a_i\ \mathrm{and}\ b_i\ \mathrm{are} \ 1\ \mathrm{and}\ 0\ \mathrm{respectively,}\\
0,&\mathrm{otherwise.}
\end{array}\right.
\\p_{i,j}^{(2)}(f)&=&\left\{
\begin{array}{ll}
f,&\mathrm{if\ the\ } j^\mathrm{th}\ \mathrm{digits\ of}\ a_i\ \mathrm{and}\ b_i\ \mathrm{are} \ 0\ \mathrm{and}\ 1\ \mathrm{respectively,}\\
0,&\mathrm{otherwise.}
\end{array}\right. \\
p_{i,j}^{(3)}(f)&=&\left\{
\begin{array}{ll}
f,&\mathrm{if\ the\ } j^\mathrm{th}\ \mathrm{digits\ of}\ a_i\ \mathrm{and}\ b_i\ \mathrm{are\ both} \ 0\ \mathrm{or\ both}\ 1,\\
0,&\mathrm{otherwise.}
\end{array}\right.
\end{eqnarray*}
\end{Lemma}
\begin{proof}
For $i=1$, we use the operator $g(j)$ from Lemma~\ref{shifts} and let $g'(j)$ denote the analogous operator which shifts $2^j$ from $b_1 $ to $a_1$.  Then it is easy to see that
\begin{eqnarray*}
p_{1,j}^{(1)}&=&g'(j)g(j)\\
p_{1,j}^{(2)}&=&g(j)g'(j)\\
p_{1,j}^{(3)}&=&1-p_{1,j}^{(1)}-p_{1,j}^{(2)}.
\end{eqnarray*}
The same proof works for the cases where $i>1$.
\end{proof}
We will sometimes drop the second subscript when it is clear from the context that it is $j$.

The next five lemmas describe in detail the generation of new
functions from existing ones in $k\Sp(V)$-submodules of $k[\PP]$. In
the proofs of some of these lemmas, we will use the
element $g_2=g_2(1)$ of $\Sp(V)$ defined in (\ref{symptransvection})
and the element
\begin{equation*}
g'_2: x_1\mapsto x_1,\quad x_2\mapsto x_1+ x_2,\quad y_1\mapsto y_1+y_2,\quad
y_2\mapsto y_2.
\end{equation*}

\begin{Lemma}\label{permutesubscripts}
Let $$f=f_0 f_1^2 f_2^{2^2}\ldots f_j^{2^j}\ldots  f_{t-1}^{2^{t-1}}$$ be an SBF of
$\HH$-type $(s_0,\ldots,s_j,\ldots,s_{t-1})$, where $f_j$ is of the class
$(\rho,\sigma,\tau,\upsilon)$.  Let $f'_j$ be any other function of the class
$(\rho,\sigma,\tau,\upsilon)$.  Then there is a group ring element $g\in k\Sp(V)$ such that
$$gf\equiv f_0 {f_1}^2 {f_2}^{2^2}\ldots {f'_j}^{2^j}\ldots  {f_{t-1}}^{2^{t-1}}$$
modulo the span  of SBFs of lower $\HH$-types.
\end{Lemma}
\begin{proof}
Clearly if $i\in T$ for $f_j$, then we can use a shift operator to
exchange  $x_i$ and $y_i$ in that digit, leave all other digits
unchanged. We need to show further that we can permute the indices
${1,\ldots,m}$ of the variables in $f_j$, while keeping the
functions for all the other digits unchanged.  It suffices to show
that we can interchange the subscripts $1$ and $2$, if they belong
to any two of $\{R,S,T,U\}$, or if they both belong to $R$. In the
calculations below there are several notational points to keep in
mind. First, we use the projectors $p_{i,j}^{(\alpha)},\
\alpha=1,2,3,$ from Lemma~\ref{project}, dropping the second
subscript when it is understood to be $j$. Next, since we are
calculating modulo the span of SBFs of lower $\HH$-types, we can
ignore any term in which there is a carry from one digit to another.
Most importantly,  we shall write only the part of the digit $f_j$
involving $x_1,\ y_1,\ x_2,$ and $y_2,$ except that when $1\in R$ we
assume it is paired with $3\in R'$ and include $x_3$ and $y_3$. We
calculate:
\begin{eqnarray*}
(1\in S,\ 2\in U)\quad g_2(x_1y_1)&=&x_1y_1+y_1x_2\\
p_1^{(2)}g_2(x_1y_1)&=&y_1x_2\\
g'_2p_1^{(2)}g_2(x_1y_1)&=&(x_1+x_2)(y_1+y_2)\\
(p_1^{(3)}g'_2p^{(2)}_1g_2-1)(x_1y_1)&=&x_2y_2
\end{eqnarray*}\begin{eqnarray*}
(1\in T,\ 2\in U)\quad g_2(x_1)&=&x_1+x_2\\
(g_2-1)(x_1)&=&x_2\\
\end{eqnarray*}\begin{eqnarray*}
(1\in S,\ 2\in T)\quad g_2(x_1y_1y_2)&=&(x_1+x_2)y_1(y_1+y_2)\\
p_1^{(2)}g_2(x_1y_1y_2)&=&y_1x_2y_2
\end{eqnarray*}\begin{eqnarray*}
(1\in R,\ 2\in U,\ 3\in R')\\ g_2(x_1y_1+x_3y_3)&=&(x_1+x_2)y_1+x_3y_3\\
p_1^{(2)}g_2(x_1y_1+x_3y_3)&=&y_1x_2\\
g'_2p_1^{(2)}g_2(x_1y_1+x_3y_3)&=&(x_1+x_2)(y_1+y_2)\\
p_1^{(3)}g'_2p_1^{(2)}g_2(x_1y_1+x_3y_3)&=&x_1y_1+x_2y_2\\
\textrm{We interchanged $1$ and $2$.}&&
\end{eqnarray*}\begin{eqnarray*}
(1\in R,\ 2\in T,\ 3\in R')\\
g_2((x_1y_1+x_3y_3)y_2)&=&((x_1+x_2)y_1+x_3y_3)(y_1+y_2)\\
&=&y_1x_3y_3+(x_1+x_2)y_1y_2+y_2x_3y_3\\
p_2^{(3)}g_2((x_1y_1+x_3y_3)y_2)&=&y_1x_3y_3+y_1x_2y_2=y_1(w_2+w_3)\\
%p_1^{(2)}p_2^{(3)}g_2((x_1y_1+x_3y_3)y_2)&=&(x_3y_3)y_1+x_2y_1y_2\\
%&=&y_1(x_2y_2+x_3y_3)
\end{eqnarray*}\begin{eqnarray*}
(1\in R,\ 2\in S,\ 3\in R')\\ g_2(x_2y_2(x_1y_1+x_3y_3))&=&
x_2(y_1+y_2)
%\\&\cdot&
((x_1+x_2)y_1+x_3y_3))
\\
&=&x_1y_1x_2y_2+
y_1x_2x_3y_3+x_2y_2x_3y_3\\
p_1^{(2)}g_2(x_2y_2(x_1y_1+x_3y_3))&=&y_1x_2x_3y_3\\
g'_2p_1^{(2)}g_2(x_2y_2(x_1y_1+x_3y_3))&=&(y_1+y_2)(x_1+x_2)x_3y_3\\
p_1^{(3)}g'_2p_1^{(2)}g_2(x_2y_2(x_1y_1+x_3y_3))&=&(x_1y_1+x_2y_2)x_3y_3.\\
\textrm{We interchanged 2 and 3.}&&
\end{eqnarray*}

We can also interchange two subscripts in $R$. Let
$$f_j=(x_1y_1+x_3y_3)(x_2y_2+x_4y_4).$$ Then
\begin{eqnarray*}
g_2(f_j)&=&((x_1+x_2)y_1+x_3y_3)(x_2(y_1+y_2)+x_4y_4)\\
p_2^{(1)}g_2(f_j)&=&y_1x_2(x_3y_3+x_4y_4)\\
p_1^{(3)}g'_2p_2^{(1)}g_2(f_j)&=&(x_1y_1+x_2y_2)(x_3y_3+x_4y_4).
\end{eqnarray*}

It is important to remember that when $g_2$ or $g'_2$ is applied,
substitutions take place in all digits, as in (\ref{linsubs}), not
just in $f_j$, and carries must be performed. In the above we have
only shown the effect of substitution on $f_j$ and, since we are
calculating modulo the span of SBFs of lower $\HH$-type, we have set
equal to zero the terms of the resulting $j^\mathrm{th}$ digit in
which there is a carry to the $(j+1)^\mathrm{th}$ digit. To see that
this calculation gives the correct $j^\mathrm{th}$ digit of the
function obtained by the substitution in all digits, as in
(\ref{linsubs}), we recall the rewriting process described in
subsection \ref{subs}. After distributing the product into a sum of
monomials there are possibly carries to be performed. The result of
a carry will be a monomial of lower $\HH$-type, which we may
disregard since we are working modulo functions  of lower
$\HH$-type. Therefore, when computing the $j^\mathrm{th}$ digit of
the function obtained by substitition modulo functions of lower
$\HH$-types, we can treat as zero any contributions from carries
between the $(j-1)^\mathrm{th}$ and $j^\mathrm{th}$ digits.

We also use the projectors to remove those terms for which the
$k^\mathrm{th}$ digit, $k\ne j$, is not equal to $f_k$. Let $f'$ be
a factorizable function of the same $\HH$-type as $f$ and with
$k^\mathrm{th}$ digit equal to $f_k$. After application of $g_2$ to
$f'$  we end up, modulo functions of lower $\HH$-type, with a
certain sum of factorizable functions of the same $\HH$-type as $f$,
including one term equal to $f'$. If neither $x_1$ nor $y_2$ occurs
in $f_k$  then the $k^\mathrm{th}$ digit of every term will be equal
to $f_k$.  Suppose either $x_1$ or $y_2$ or $x_1y_2$ is a factor of
$f_k$. We treat as zero any term involving a carry from the
$(k-1)^{\rm th}$ to the $k^{\rm th}$ digit or from the $k^{\rm th}$
to the $(k+1)^{\rm th}$ digit. Then there is some choice of
$\alpha_k$, $1\leq \alpha_k\leq 3$, such that $p_{1k}^{(\alpha_k)}$
acts as the identity on $f'$ and kills any function in the above sum
whose $k$-th digit differs from $f_k$ in the variables $x_1$ and
$y_1$. After applying these projectors for all $k\neq j$, we end up
with a sum of factorizable functions whose $k^\mathrm{th}$  digits
are the same as those of $f$, for all $k\neq j$. Similarly, every
function arising from the application of $g'_2$ whose
$k^\mathrm{th}$ digit is unequal to $f_k$, for some $k\neq j$, can
be removed by a suitable projector.
%%%%%%%%%%%%%%%%%%%%%%%%%%%%%%%%%%%%%%%%%%%%%%%%%%%%
%  We get the original function back by taking $x_1$ from
%$x_1\mapsto x_1+x_2$ and $y_2$ from $y_2\mapsto y_1+y_2$.  We apply the projector
%$p_{1,k}^{(1)},\ p_{1,k}^{(2)},$ or $p_{1,k}^{(3)}$, according to which one preserves $f_k$.  If the term is produced from the $x_2$ of $x_1\mapsto x_1+x_2$, since there is no carry, there must be $x_2$ but no $x_1$ in the new function for the $k^\mathrm{th}$ digit, while there must have been $x_1$ but no $x_2$ in the original $f_k$.  Therefore $p_{1,k}^{(1)}$ acts as the identity on $f_k$ but kills the new term.  We similarly make sure that no new terms occur in $f_k$ with respect to $y_1$ and $y_2$.
We have shown that we can change $f_j$ to any other SBF of the same class, while keeping every other digit the same.
\end{proof}

\begin{Lemma}\label{R-up}
As in Lemma~\ref{permutesubscripts}, let $f_j$ be the $j^\mathrm{th}$ digit of an SBF $f$, and let the class of $f_j$ be $(\rho,\sigma,\tau,\upsilon),\ \tau\ge2$.  Then there exists a group ring element $g$ such that $gf$ is an SBF identical to $f$ (modulo the span of SBFs of lower $\HH$-types), except that $f_j$ is replaced by a function of class $(\rho+1,\sigma,\tau-2,\upsilon)$.
\end{Lemma}
\begin{proof} By Lemma~\ref{permutesubscripts}, we may assume that $\{1,2\}\subseteq T$ and
$f_j=x_1y_2\cdots$. Also by Lemma~\ref{permutesubscripts}, we can
apply a group ring element of $k{\rm Sp}(V)$ to $f$ to obtain a
function $f'$ which is identical to $f$ except that $f_j$ is
replaced by $f'_j=y_1x_2\cdots$, where $f_j$ and $f'_j$ differ only
in the first two subscripts.  Then we have
$$g_2f_j-f_j-f'_j=[(x_1+x_2)(y_1+y_2)-f_j-f'_j]\cdots=(x_1y_1+x_2y_2)\cdots.$$  We apply projectors to the other digits, as in the proof of Lemma~\ref{permutesubscripts}, to kill any functions which are not the same as the original ones in those digits.
\end{proof}

\begin{Lemma}\label{T-up}
Let $f_j$ be the $j^\mathrm{th}$ digit of an SBF $f$, and let the class of $f_j$ be $(\rho,\sigma,\tau,\upsilon)\ne (1,0,m-2,0),\ \rho>0$.  Then there exists a group ring element $g$ such that $gf$ is an SBF identical to $f$ modulo the span of SBFs of lower $\HH$-types, except that $f_j$ is replaced by a function of class
$(\rho-1,\sigma,\tau+2,\upsilon)$.
\end{Lemma}
\begin{proof}
Since $(\rho,\sigma,\tau,\upsilon)\ne (1,0,m-2,0)$, and
$2\rho+\sigma+\tau+\upsilon=m$, we will consider three cases:
$\rho>1,\ \sigma>0$, and $\upsilon>0$.  We omit the part of $f_j$
which is not acted upon by $g$. Note that by
Lemma~\ref{permutesubscripts}, we may take $f_j$ to have some
special form, as long as it belongs to the class
$(\rho,\sigma,\tau,\upsilon)$.

If $\rho>1$, we take $f_j=(x_1y_1+x_3y_3)(x_2y_2+x_4y_4)$ and we get (keeping only the square-free terms):

\begin{eqnarray*}
(g_2-1)f_j&=&((x_1+x_2)y_1+x_3y_3)(x_2(y_1+y_2)+x_4y_4)-f_j\\
&\equiv&y_1x_2(x_3y_3+x_4y_4).
\end{eqnarray*}

If $\rho>0,\ \sigma>0$, then we take $f_j=x_1y_1(x_2y_2+x_3y_3)$ and compute
\begin{eqnarray*}
(g_2-1)f_j&=&(x_1+x_2)y_1(x_2(y_1+y_2)+x_3y_3)-f_j\\
&\equiv&(x_1y_1x_2y_2+(x_1+x_2)y_1x_3y_3)-f_j\\
&=&y_1x_2x_3y_3.
\end{eqnarray*}

If $\rho>0,\ \upsilon>0$, we take $f_j=(x_1y_1+x_3y_3)$ and similarly get
\begin{eqnarray*}
(g_2-1)f_j&=&((x_1+x_2)y_1+x_3y_3)-f_j\\
&=&y_1x_2.
\end{eqnarray*}

We again keep the functions for the other digits from changing as in the proof of
Lemma~\ref{permutesubscripts}.
\end{proof}

\begin{Lemma}\label{get-m}
Let $f$ be an SBF whose $j^\mathrm{th}$ digit is $f_j$ of class
$(0,1,m-2,1)$.  Then there is a group ring element $g$ such that $gf$ is identical to $f$, modulo the span of SBFs of lower $\HH$-types, except that $f_j$ is replaced by a function of class $(0,0,m,0)$.
\end{Lemma}
\begin{proof} By Lemma~\ref{permutesubscripts}, we may take $1\in S,\ 2\in U$, so that
$f_j=x_1y_1(x_2y_2)^0\cdots$.  Then $(g_2-1)f_j=y_1x_2\cdots$, and
we proceed as before.
\end{proof}

\begin{Lemma}\label{S-down}
Let $f$ be an SBF whose $j^\mathrm{th}$ digit is $f_j$ of class
$(\rho,\sigma,\tau,\upsilon)$, $\sigma>0,\ \upsilon>0$.  Then there is a group ring element $g$ such that $gf$ is identical to $f$, modulo the span of SBFs of lower $\HH$-types, except that $f_j$ is replaced by a function of class $(\rho+1,\sigma-1,\tau,\upsilon-1)$.
\end{Lemma}
\begin{proof} By Lemma~\ref{permutesubscripts}, we may assume $1\in S$ and $2\in U$.  Let $h$ be the element from
Lemma~\ref{permutesubscripts} which interchanges $1$ and $2$ in
$f_j$, while keeping the rest of $f$ the same.  Then $g=1+h$ is the
desired element. \end{proof}

The next lemma, whose precise statement is rather complicated,
relates SBFs of a given $\HH$-type with SBFs whose $\HH$-types
are one step down in the partial order on $\HH$. A fuller
discussion of its meaning will be given after the proof.

\begin{Lemma}\label{HH-down}
Let  $(s_0,s_1,\ldots,s_j,\ldots,s_{t-1})$ and
$(s_0,s_1,\ldots,s_j-1,\ldots,s_{t-1})$ be a pair of $\HH$-types,
with $s_k\le 2m-r,\ 0\le k\le t-1$, and let $\lambda_j=2s_{j+1}-s_j$
for $0\le j\le t-1$. Suppose $f$ is an SBF of type
$(s_0,s_1,\ldots,s_j,\ldots,s_{t-1})$. Assume that the
$j^{\mathrm{th}}$ digit $f_j$ of $f$ is of class
$(0,\sigma_j,\tau_j,\upsilon_j)$, where
$$
\begin{aligned}
\sigma_j&=\min\{2m-r-s_j + \delta(r\le m)(m-r),\lfloor{\lambda_j}/2\rfloor\},\\
\tau_j&=\lambda_j-2\sigma_j, \text{and}\\
\upsilon_j&=m-\sigma_j-\tau_j.
\end{aligned}
$$
Case 1. $\sigma_j\neq\lambda_j/2$ (so that $\tau_j>0$).
\begin{enumerate}[{\rm(}a{\rm)}]
\item If  $\sigma_{j-1}<\lfloor\frac{\lambda_{j-1}}2\rfloor$, let $f_{j-1}$ be of class
$(1,\sigma_{j-1},\tau_{j-1},\upsilon_{j-1})$, where
\begin{eqnarray}\label{sigmatauupsilon}
\sigma_{j-1}&=&2m-r-s_{j-1} + \delta(r\leq m)(m-r),\notag\\
\tau_{j-1}&=&\lambda_{j-1}-2\sigma_{j-1}-2,\notag\\
\upsilon_{j-1}&=&m-\sigma_{j-1}-\tau_{j-1}-2.
\end{eqnarray}

\noindent Then the following is true modulo the span of SBFs of
lower $\HH$-type. There is a group ring element $g$ such that the
$j^{\mathrm{th}}$ digit of $gf$ is of class
$(0,\sigma_j+1,\tau_j-1,\upsilon_j)$ and the $(j-1)^{\mathrm{th}}$
digit of $gf$ is of class
$(0,\sigma_{j-1},\tau_{j-1},\upsilon_{j-1}+2)$, while all the other
digits remain the same.

\item If  $\sigma_{j-1}=\lfloor\frac{\lambda_{j-1}}2\rfloor$, let
$f_{j-1}$ be of class $(0,\lambda_{j-1}/2,0,m-\lambda_{j-1}/2)$ or
$(0,(\lambda_{j-1}-1)/2,1,m-(\lambda_{j-1}+1)/2)$, depending on
whether $\lambda_{j-1}$ is  even or odd. Then the following is true
modulo the span of SBFs of lower $\HH$-type. There is a group ring
element $g$ such that $gf$ has $j^{\mathrm{th}}$ digit as in (a) and
$(j-1)^{\mathrm{th}}$ digit of class
$(0,\sigma_{j-1}-1,\tau_{j-1},\upsilon_{j-1}+1)$, while all the
other digits remain the same.

\end{enumerate}

\noindent Case 2. $\sigma_j=\lambda_j/2$.

\begin{enumerate}[{\rm(}a{\rm)}]
\item If  $\sigma_{j-1}<\lfloor\frac{\lambda_{j-1}}2\rfloor$, let $f_{j-1}$ be of class
$(1,\sigma_{j-1},\tau_{j-1},\upsilon_{j-1})$, where
$\sigma_{j-1},\tau_{j-1},\upsilon_{j-1}$ are defined as in
(\ref{sigmatauupsilon}). Then the following is true modulo the span
of SBFs of lower $\HH$-type. There is a group ring element $g$ such
that the $j^{\mathrm{th}}$ digit of $gf$ is of class
$(0,\frac{\lambda_j}{2},1,m-\frac{\lambda_j}{2}-1)$ and the
$(j-1)^{\mathrm{th}}$ digit of $gf$ is of class
$(0,\sigma_{j-1},\tau_{j-1},\upsilon_{j-1}+2)$, while all the other
digits remain the same.

\item If  $\sigma_{j-1}=\lfloor\frac{\lambda_{j-1}}2\rfloor$, let
$f_{j-1}$ be of class $(0,\lambda_{j-1}/2,0,m-\lambda_{j-1}/2)$ or
$(0,(\lambda_{j-1}-1)/2,1,m-(\lambda_{j-1}+1)/2)$, depending on
whether $\lambda_{j-1}$ is  even or odd. Then the following is true
modulo the span of SBFs of lower $\HH$-type. There is a group ring
element $g$ such that $gf$ has $j^{\mathrm{th}}$ digit as in (a) and
$(j-1)^{\mathrm{th}}$ digit of class
$(0,\sigma_{j-1}-1,\tau_{j-1},\upsilon_{j-1}+1)$, while all the
other digits remain the same.
\end{enumerate}

\noindent In both Case 1 and Case 2, the elements $g$ can be chosen
so that if $f'$ is a monomial of $\HH$-type lower than that of $f$,
then $gf'$ is either zero or has $\HH$-type lower than that of $gf$.

\end{Lemma}

\begin{proof} For $\HH$-type $(s_0,\ldots ,s_{j-1},s_j-1,\ldots,s_{t-1})$, we let
$(\lambda'_0,\lambda'_1,\ldots ,\lambda'_{t-1})$ be the
corresponding type in $\Llambda$. Since $(s_0,\ldots
,s_{j-1},s_j-1,\ldots,s_{t-1})$ is the $\HH$-type that results when
there is a carry from the $(j-1)^{\rm th}$ digit to the $j^{\rm th}$
digit of $f$, we have $\lambda'_{j-1}=\lambda_{j-1}-2\geq 0$,
$\lambda'_j=\lambda_j+1$, and $\lambda'_k=\lambda_k$ for all
$k\not\in\{j-1, j\}$.

We deal with Case 1 first. We may assume that $1\in R$ for $f_{j-1}$
in Case 1(a), or that $1\in S$ for $f_{j-1}$ in Case 1(b). In either
case, we may assume that $1\in T$ for $f_j=x_1{y_1}^0\ldots$, using
Lemma~\ref{permutesubscripts}, at the cost of adding some terms in
the span of SBFs of lower $\HH$-type than $f$.  The shift operator
$g(j-1)$ shifts $2^{j-1}$ from $a_1$, the exponent of $x_1$, to
$b_1$, the exponent of $y_1$, or returns $0$ if a monomial has a $0$
in the $(j-1)^\mathrm{th}$ digit of $a_1$.  It is clear that
$g(j-1)f$ has the desired properties.

For Case 2, we may make the same assumptions on $f_{j-1}$ as in Case
1. Note that $\sigma_j=\frac{\lambda_j}{2}$ and
$m=\sigma_j+\upsilon_j$, we have $\upsilon_j>0$ (otherwise
$\lambda_j=2m$, which implies $\lambda'_j=2m+1$, impossible).
Therefore for $f_j$, we can assume that $1\in U$. Let $g(j-1)$ be
the same shift operator as above. We see that $g(j-1)f$ has the
desired properties.

Suppose $f'$ is a monomial of $\HH$-type $\sss'$ strictly below
$\sss=(s_0,\ldots,s_{t-1})$. Then $p_{1,j-1}^{(3)}f'$ is equal to
$0$, if the $(j-1)^{\mathrm{th}}$ digits of the exponents of $x_1$
and $y_1$ are different, or to $f'$ if these exponents are the same.
In the latter case either $g(j-1)f'$ is  zero or else its $\HH$-type
$\sss''$ is strictly less than $\sss'$ due to carry in its
computation. Since the ordering on $\HH$ reflects the
$k\Sp(V)$-module structure (\S\ref{substruct}), we also know
$\sss''\leq (s_0,\ldots,s_j-1,\ldots,s_{t-1})$. We see that this
inequality must actually be strict, because the sum of the entries
of $\sss''$ is at least $2$ less than $\sum_{j=0}^{t-1}s_j$. Thus,
since $p_{1,j-1}^{(3)}f=f$, the group ring element
$g(j-1)p_{1,j-1}^{(3)}$ has the required properties in the last
statement of the lemma.
\end{proof} \vspace{0.1in}

We would like to explain the meaning of this rather technical lemma.
We first focus on the $(j-1)^{\rm th}$ digit of $f$ and $gf$. Note
that the maximum level for the $(j-1)^{\rm th}$ digit of an
$r$-admissible function of $\HH$-type $(s_0,\ldots
,s_{j-1},s_j,\ldots,s_{t-1})$ is $L_{j-1}:={\rm min}\{\ell_{j-1},
\lfloor\frac{\lambda_{j-1}}{2}\rfloor\},$ where
$\ell_{j-1}=2m-r-s_{j-1}+\delta(r\leq m)(m-r)$. In the lemma, we
have set $\sigma_{j-1}$ equal $L_{j-1}$. So the level $\sigma_{j-1}$
of $f_{j-1}$ has its maximum value for the given $\HH$-type. The
maximum level for the $(j-1)^{\rm th}$ digit of an $r$-admissible
function of $\HH$-type $(s_0,\ldots ,s_{j-1},s_j-1,\ldots,s_{t-1})$
is $L'_{j-1}:={\rm min}\{\ell_{j-1},
\lfloor\frac{\lambda'_{j-1}}{2}\rfloor\}$. Since
$\lambda'_{j-1}=\lambda_{j-1}-2$, we have
$$L'_{j-1}= \left\{\begin{array}{ll}
    \sigma_{j-1},\; & {\rm if}\; \sigma_{j-1}<\lfloor\frac{\lambda_{j-1}}{2}\rfloor, \\
    \sigma_{j-1}-1,\; & {\rm if}\;
    \sigma_{j-1}=\lfloor\frac{\lambda_{j-1}}{2}\rfloor.
            \end{array}\right. $$
So the level of the $(j-1)^\mathrm{th}$ digit of $gf$ also has its
maximum value for that given $\HH$-type.

In the lemma, we also specified $\sigma_j$ to have its maximum
value. The maximum level for the $j^{\rm th}$ digit of an
$r$-admissible function of $\HH$-type $(s_0,\ldots
,s_{j-1},s_j-1,\ldots,s_{t-1})$ is $$L'_{j}:={\rm
min}\{2m-r-s_{j}+1+\delta(r\leq m)(m-r),
\lfloor\frac{\lambda'_{j}}{2}\rfloor\}.$$ In Case 1 of the lemma,
since $2\sigma_j<\lambda_j$ and $\lambda'_j=\lambda_j+1$, we have
$L'_j=\sigma_j+1$. So the level $\sigma_j+1$ of the $j^\mathrm{th}$
digit of $gf$ in Case 1 of the lemma is maximum for that given
$\HH$-type. In Case 2 of the lemma, since $\lambda_j=2\sigma_j$ and
$\lambda'_j=\lambda_j+1$, we have $L'_j=\sigma_j$. So the level
$\sigma_j$ of the $j^\mathrm{th}$ digit of $gf$ in Case 2 of the
lemma is also maximum for that given $\HH$-type.

Thus, assuming we picked $\sigma_k$ also to be maximum for each of
the other digits, $0\le k\le t-1$, $gf$ will be an SBF of $\HH$-type
$(s_0,\ldots,s_j-1,\ldots,s_{t-1})$ which has the maximum
$r$-admissible level for each digit, for that $\HH$-type.

\begin{Lemma}
\label{generation2} %Main generation lemma.
 The set of $r$-admissible SBFs forms a basis for $\CC_r$.
\end{Lemma}
\begin{proof}
The set of SBFs is linearly independent by construction.
By Lemma~\ref{Cinside}, it is sufficient to prove that
every $r$-admissible SBF lies in $\CC_r$.
We shall proceed by induction on the partial order on $\HH$.
To start the induction, we observe that
$\CC_r\supseteq\CC_{2m-1}$, by  (\ref{Cnested}), and that
$\CC_{2m-1}$ contains the constant functions
and all $r$-admissible SBFs of the lowest
$\HH$-type  $(1,1,\dots,1)$, since it contains all basis monomials of this type.
Next we wish to explain how Lemmas~\ref{permutesubscripts}--\ref{HH-down}
will be applied in our inductive proof. They are all of the same form:
from a given function $f$ and a target function $f'$ of a certain $\HH$-type,
each lemma yields a group ring element $g$ such that $f'-gf$ lies in the span
of SBFs whose $\HH$-types are lower than that of $f'$. If the functions
$f$ and $f'$ belong to $\MM_r$, then by Lemma~\ref{submodule lemma}
the conclusions of Lemmas~\ref{permutesubscripts}--\ref{HH-down} can be strengthened to  say that the error term $f'-gf$
is in the span of $r$-admissible  SBFs of lower $\HH$-type than $f'$.
With this in mind, we can now continue with the proof.

Let $f'$ be an $r$-admissible SBF. We will show that $f'\in\CC_r$.
We can assume that $f'$ is not constant and has $\HH$-type strictly
higher than $(1,1,\dots, 1)$. The inductive hypothesis states
that every $r$-admissible $SBF$ of strictly lower $\HH$-type
belongs to $\CC_r$.

Suppose  $\CC_r$ contains any SBF whose digits belong to the same classes
as the corresponding digits of $f'$. By Lemma~\ref{Cinside}
such an SBF belongs to $\MM_r$. Then
Lemma~\ref{permutesubscripts} shows, together with the inductive
hypothesis, that $f'\in\CC_r$.
So it suffices to show that $\CC_r$ contains some
SBF with the same digit classes as $f'$.

Next, let $f$ be an SBF of the same $\HH$-type as $f'$, but having
maximum $r$-admissible levels in all digits. We will show that the
$f'$ lies in the $k\Sp(V)$-submodule generated by  $f$ and $\CC_r$.
This will go smoothly unless very special conditions hold, in which
case more work is required to obtain the result. By the inductive
hypothesis, it is enough to show  that the above submodule contains
$f'$ modulo the span of $r$-admissible SBFs of lower $\HH$-type. We
do this by applying the group ring operators given by
Lemmas~\ref{R-up}--\ref{S-down}. These operators allow us to modify
the digits of $f$ one at a time while leaving the other digits
unchanged until we obtain, modulo the span of $r$-admissible of
lower $\HH$-type, an SBF with the same digit classes as $f'$. Then
we apply the preceding paragraph.

Of course, we only need to modify a digit of $f$ if it belongs to a
different class from the corresponding digit of $f'$. If the
$j^{\mathrm{th}}$ digit of $f'$ is not of class $(0,0,m,0)$, then
using Lemmas~\ref{R-up}, \ref{T-up} and \ref{S-down} we can modify
the $j^{\mathrm{th}}$ digit of $f$ successfully. If the $j^{\mathrm{th}}$
digit of $f'$ is $(0,0,m,0)$ but the maximum $r$-admissible
level for the $j^{\mathrm{th}}$ digit in an SBF of this $\HH$-type is
nonzero, then the same lemmas, together with Lemma~\ref{get-m} allow
us to modify the $j^{\mathrm{th}}$ digit $f$ to the desired form.
Finally, suppose the maximum $r$-admissible level $\ell_j$ for the
$j^{\mathrm{th}}$ digit is equal to $0$. Let $\HH$-type of $f'$ and
$f$ be $(s_0,\ldots ,s_{t-1})$. Then the conditions (\ref{r-admit})
$$
s_j\le 2m-r, \quad \ell_j=(m-r)\delta(r\leq m) + (2m-r-s_j)
$$
force $r\ge m$ and $s_j=2m-r$. Then
\begin{equation}\label{meq}
m=\lambda_j=2s_{j+1}-s_j\leq 2(2m-r)-(2m-r)=2m-r\leq m;
\end{equation}
so equality holds throughout, and $s_j=s_{j+1}=2m-r=m$, which means
$r=m$. If $r=m$ and some digits of $f'$ have class $(0,0,m,0)$ and
$0$ is the maximum $r$-admissible level for those digits, we shall
refer to $f'$ as \emph{exceptional}. Thus, unless $f'$ is
exceptional, the desired conclusion holds.

Suppose the $\HH$-type $\sss'$ of $f'$ is not
$(2m-r,2m-r,\ldots,2m-r)$, the highest possible $\HH$-type for
$r$-admissible SBFs.  Then Lemma~\ref{HH-down} says that for any
$\HH$-type $\sss''$ which is one step higher than that of $f'$ in
the the partial order, and any $r$-admissible SBF $f''$ of type
$\sss''$ which has maximum $r$-admissible levels in all digits,
there exists a group ring element $g''$, such  $g''f''$ is an SBF of
type $\sss'$ and maximum $r$-admissible levels in all digits, modulo
the span of $r$-admissible SBFs of $\HH$-types lower than $\sss'$.
Moreover, $g''$  sends monomials of $\HH$-type lower than $\sss''$
to monomials of $\HH$-type lower than $\sss'$.

Given two $\HH$-types, ${\bf s}=(s_0,\ldots,s_{t-1})$ and ${\bf
s}'=(s'_0,\ldots,s'_{t-1}),\ 2m-r\ge s_j\ge s'_j,\
0\le j\le t-1,$ it is easy to see that there exists a sequence of
$\HH$-types connecting $\sss$ and $\sss'$, such that two
consecutive $\HH$-types differ in only one component, with that
component of the second being reduced by $1$ (see
\cite{bardoe-sin}).

Consider the characteristic function (\ref{charfunctioneq}), which belongs
to $\CC_r$ by definition. It is the sum of its
leading monomial
$$
m_+=\begin{cases}x_1^{q-1}\cdots x_m^{q-1}y_1^{q-1}\cdots y_{m-r}^{q-1},\quad
\text{if $r<m$,}\\
x_1^{q-1}\cdots x_{2m-r}^{q-1},\quad\text{if $r\geq m$}\end{cases}
$$
and some other $r$-admissible SBFs of lower $\HH$-type. Now $m_+$ is
an SBF of type $(2m-r,\ldots,2m-r)$ and has maximum $r$-admissible
levels in all of its digits. Thus, by repeated application of
Lemma~\ref{HH-down} down a sequence from $(2m-r,\ldots,2m-r)$ to
$\sss'$, we see that $\CC_r$ contains an $r$-admissible SBF $f$ of
type $\sss'$, modulo the span of $r$-admissible SBFs of $\HH$-types
lower than $\sss'$. If $f'$ is not exceptional, then we have seen
that $\CC_r$ also contains $f'$, modulo the span of $r$-admissible
SBFs of lower $\HH$-type.

If $f'$ is exceptional, then $r=m$ and each exceptional digit
of $f'$ belongs to the same class $(0,0,m,0)$ as the same digit of $m_+$.
We can use Lemmas~\ref{R-up}, \ref{T-up}, \ref{S-down} and
\ref{HH-down} to modify the other digits of $m_+$, without
doing anything to the exceptional digits, until we obtain
$f'$ modulo the span of $r$-admissible SBFs of lower $\HH$-type.

In both cases, we apply the inductive hypothesis to deduce
$f'\in\CC_r$, which completes the proof.
\end{proof}

%%%%%%%%%%%%%%%%%%%%%%%%%%%%%%%%%%%%%%%%%%%%%%%%%%%%%%%%%%%%%%%%%%%%%%%

\section{The case $q=2$}\label{q=2}

In this section we modify the arguments of Section~\ref{admiss_bases} to show that
Theorem~\ref{basis_thm} also holds when $q=2$. We will use the
$\Sp(V)$ elements  $g_1=g_1(1)$ and  $g_2=g_2(1)$ defined
in (\ref{transvection}) and (\ref{symptransvection}).
\begin{Lemma}
Lemma~\ref{submodule lemma} holds when $q=2$.
\end{Lemma}
\begin{proof}
We again examine the actions of $g_1$ and $g_2$, since the result is trivial for the other generators of $\Sp(2m,2)$. We now must do the calculations modulo
$(x_2^2-x_2, y_1^2-y_1)$, instead of modulo $(x_2^2-1, y_1^2-1)$.  Instead of redoing the calculations, we make the following observation:  when $q=2$, replacing the squared variable in a function by the unsquared one is equivalent to the situation, in the $q=2^t>2$ case, of a carry from the $0^\mathrm{th}$ digit to the $1^\mathrm{st}$ digit, replacing the squared variable in the $0^\mathrm{th}$ digit by $1$, followed by a carry from the $(t-1)^\mathrm{th}$ digit into the  $0^\mathrm{th}$ digit.  Therefore, the calculations in the proof of
Lemma~\ref{submodule lemma} also prove this lemma for $q=2$.
\end{proof}
%\begin{Remark}
%Notice that we could regard the above calculations as combining the cancellation of square terms, as
%before, combined with a carry and reduction of $s$.  Maybe the explicit calculations need not be repeated.
%\end{Remark}

Here we comment that Lemma~\ref{shifts} does not hold if $q=2$, because the elements of $\F_2$ sum to $1$,
while the elements of any other finite field sum to $0$.  Thus the projectors developed in Lemma~\ref{project} are not available, and we have the need for this section.
\begin{Lemma}\label{permuteq2} The following lemmas from Section~\ref{admiss_bases} hold when $q=2$: (a) Lemma~\ref{permutesubscripts},
(b) Lemma~\ref{R-up}, (c) Lemma~\ref{get-m}, (d) Lemma~\ref{S-down}.
 Furthermore, when $q=2$, we may drop the condition, ``modulo the span of SBFs of lower $\HH$-type'' in all cases.
\end{Lemma}
\begin{proof}
 For Lemma~\ref{permutesubscripts}, since there is only one
digit, we can just permute the subscripts directly using group elements of the form (\ref{gen2}), and we can interchange $x_i$ and $y_i,\ 1\le i\le m$ by combining elements of the form (\ref{gen2}) with (\ref{gen3}).
In the cases of Lemmas~\ref{R-up},  \ref{get-m}, and \ref{S-down}, we do not use any shift operators in the proofs.  Also, the projectors
are used only to keep digits other than the $j^\mathrm{th}$ digit from changing, but now we have only one digit.
Thus the proofs are still valid for the case $q=2$.
\end{proof}

The $q=2$ version of Lemma~\ref{T-up} requires a little more work to prove.
\begin{Lemma}\label{2T-up}
Let  $f$ be  an SBF, and let the class of $f$ be $(\rho,\sigma,\tau,\upsilon)\ne (1,0,m-2,0),\ \rho>0$.  Then there exists a group ring element $g$ such that $gf$ is  of class
$(\rho-1,\sigma,\tau+2,\upsilon)$.
\end{Lemma}

\begin{proof}
Since $(\rho,\sigma,\tau,\upsilon)\ne (1,0,m-2,0)$, and
$2\rho+\sigma+\tau+\upsilon=m$, we will consider three cases:
$\rho>1,\ \sigma>0$, and $\upsilon>0$.  In each case, by Lemma~\ref{permuteq2}(a), we may take $f$ to have the form indicated.  We omit the part of $f$ which is not acted on by $g$.

If $\rho>1$, we may assume $f=(x_1y_1+x_3y_3)(x_2y_2+x_4y_4)$ and we get:

\begin{eqnarray*}
(g_2-1)f&=&((x_1+x_2)y_1+x_3y_3)(x_2(y_1+y_2)+x_4y_4)-f\\
&=&y_1x_2(w_3+w_4)+y_1x_2+y_1w_2+w_1x_2.
\end{eqnarray*}
Next we interchange $x_1$ and $y_1$ and apply $(g_1-1)$ to get
\begin{eqnarray*}
(g_1-1)(x_1x_2(w_3+w_4)+x_1x_2+x_1w_2+w_1x_2)&=&y_1x_2(w_3+w_4)+y_1w_2.
\end{eqnarray*}
Then, interchanging subscripts $1$ and $2$ and again applying $(g_1-1)$, we get
\begin{eqnarray*}
(g_1-1)(x_1y_2(w_3+w_4)+w_1y_2)&=&y_1y_2(w_3+w_4)+y_1y_2.
\end{eqnarray*}

If $\rho>0,\ \sigma>0$, then we may assume $f=x_1y_1(x_2y_2+x_3y_3)$ and compute
\begin{eqnarray*}
(g_2-1)f&=&(x_1+x_2)y_1(x_2(y_1+y_2)+x_3y_3)-f\\
&=&(w_1w_2+(x_1+x_2)y_1w_3)+w_1x_2+y_1w_2+y_1x_2-f\\
&=&y_1x_2w_3+w_1x_2+y_1w_2+y_1x_2.
\end{eqnarray*}

If $\rho>0,\ \upsilon>0$, we may assume $f=(x_1y_1+x_3y_3),\ 2\in U$ and similarly get
\begin{eqnarray*}
(g_2-1)f&=&((x_1+x_2)y_1+x_3y_3)-f\\
&=&y_1x_2.
\end{eqnarray*}

Now we notice that for $1<i\le m,\ (g_1-1)(w_1+w_i)= y_1$.  Therefore, we can always go from a function of class $(\rho,\sigma,\tau,\upsilon),\ \rho\ne0,$ to one of class $(\rho-1,\sigma,\tau+1,\upsilon+1)$, reducing the degree $\lambda$ by $1$.  Similarly, since $(g_1-1)w_1\equiv y_1$, we can go from a function of class $(\rho,\sigma,\tau,\upsilon),\ \sigma>0$, to one of class
$(\rho,\sigma-1,\tau+1,\upsilon)$.  Therefore, in the first case, where $\rho>1$, by Lemma~\ref{permuteq2}(a), we may subtract the extra term of degree $2$.  In the second case, where $\sigma>0$, we may subtract the extra term of degree $2$ and the two terms of degree $3$.  We have exactly the functions we want.
\end{proof}

  The following lemma takes the place of Lemma~\ref{HH-down}.
It shows that given an SBF of degree $\lambda=s>1$ and of a class having maximum
$\sigma$ for that degree, we can find a group ring element which will give us an SBF
with $\lambda$ reduced by $1$ and $\sigma$ maximum for the new degree.
\begin{Lemma}\label{2HH-down}
Let $s\le 2m-r$ and let $f$ be an SBF of class $(0,\sigma,\tau,\upsilon)$ of degree
$\lambda=s>1$, where
$\sigma=\min\{2m-r-s+\delta(m>r)(m-r),\lfloor s/2\rfloor\},\
\tau=s-2\sigma,$ and $\upsilon=m-\sigma-\tau$.
\begin{enumerate}[\rm(a)]
\item If $\sigma<s/2-1$, then there exists
$g\in k[\Sp(V)]$ such that $gf$ is of class
$(0,\sigma+1,\tau-3, \upsilon+2)$.
\item  If $\sigma=s/2-1$ or
$\sigma=(s-1)/2$, then there is a $g\in k[\Sp(V)]$ such that
$gf$ is of class $(0,\sigma,\tau-1,\upsilon+1)$.
\item If $\sigma=s/2$, then there is a $g\in
k[\Sp(V)]$ such that $gf$ is of class
$(0,\sigma-1,\tau+1,\upsilon)$.
\end{enumerate}
\end{Lemma}
\begin{proof}
\begin{enumerate}[\rm(a)]
\item In the  first case, by Lemma~\ref{permuteq2}(b), we instead consider $f'$ and $f^{''}$ of
class
$(1,\sigma,\tau-2, \upsilon)$ of the forms
$f'=(w_1+w_3)y_2\cdots$ and $f^{''}=(w_1+w_3)x_2\cdots$.
We first calculate:
$$(g_2-1)f'=y_1x_2\cdots.$$
  We also have (omitting variables with subscripts higher than $3$):
\begin{eqnarray*}
(g_1+g_2)f'&=&((x_1+y_1)y_1+x_3y_3)y_2+((x_1+x_2)y_1+x_3y_3)(y_1+y_2)\\
&=&y_1(w_2+w_3)+w_1+y_1y_2+x_2y_1.
\end{eqnarray*}
Since the first term is of the same class we started with, and the last two are of
the same class we obtained from $f^{''}$, we may subtract
them, leaving only $w_1$, the term we want.

\item If $\sigma=s/2-1,\ \tau=2$, by Lemma~\ref{permuteq2}(a), we may take $f=x_1x_2\cdots$, and get
$(g_2-1)f=x_2\cdots$.

\noindent If $\sigma=(s-1)/2>0$,  we may take $f=w_1y_2\cdots$ and get (omitting variables
with subscripts higher than $2$)
$(g_2-1)f=y_1w_2+w_1+y_1x_2$.  We may subtract the first term, since it is of the
same class as $f$.  We subtract the last term, as it is of the same class as
$(g_1-1)f=y_1y_2$, leaving the term we want, $w_1$.

\item If $\sigma=s/2$, we take $f=w_1\cdots$ and have $(g_1-1)f=y_1\cdots$.
\end{enumerate}\end{proof}

The following lemma gives us an SBF which generates all the other
SBFs in the submodule, according to the lemmas of this section.
\begin{Lemma}
If $r\ge m$, there is an SBF of class $(0,0,2m-r,r-m)$ in $\CC_r$. If
$r\le m$, there is an SBF of class $(0,m-r,r,0)$ in $\CC_r$.
\end{Lemma}
\begin{proof}
We may assume $1<r<2m-1$, since the case $r=1$ is trivial and the
case $r=2m-1$ is the incidence of points with all hyperplanes
\cite{CSX1}.
We will pick a particular $r$-subspace $W$ and show how
to obtain the desired monomial SBF by applying
a sequence of group ring elements.
If $r\le m$, we take the totally isotopic $r$-subspace
$x_1=0$, \dots, $x_r=0$, $x_{r+1}=0$, $y_{r+1}=0$, \dots, $x_{m}=0$,
$y_{m}=0$,
with characteristic function
\begin{equation*}
f=(x_1-1)(x_2-1)\cdots(x_r-1)(x_{r+1}-1)
(y_{r+1}-1)\cdots(x_{m}-1)(y_{m}-1),
\end{equation*}
while if $r\ge m$, we pick
the $r$-subspace $x_1=0$, \dots, $x_{2m-r}=0$, which
is the complement of a totally isotropic subspace with respect to the form.
Its characteristic function is
\begin{equation*}
f=(x_1-1)(x_2-1)\cdots(x_{2m-r}-1).
\end{equation*}
Let $H$ be the $(r+2)$-dimensional subspace of $V$
defined by the same equations as $W$, but with the
two equations $x_1=0$ and $x_2=0$ omitted and let
$f_H$ be its characteristic function. We have
\begin{equation*}
f=(x_1-1)(x_2-1)f_H.
\end{equation*}
By (\ref{Cnested}), we have $f_H\in\CC_r$.
Next let $f'=(x_1-1)(y_2-1)f_H$ be the function obtained from
$f$ by interchanging $x_2$ and $y_2$.  Then $f'\in\CC_r$, since $\CC_r$ is a
$k\Sp(V)$-module.
We calculate:
\begin{equation*}
(g_2-1)f'= [x_1y_1+x_2y_2+1+(y_1-1)(x_2-1)]f_H.
\end{equation*}
We may subtract the last term because it represents the
characteristic function of another $r$-subspace.  Then we calculate:
\begin{equation*}
(g_1-1)(x_1y_1+x_2y_2+1)f_H=y_1f_H.
\end{equation*}
Using these results, we see that
$x_1f_H$, $y_2f_H$ and $f_H$ lie in $\CC_r$. If we subtract these
terms from $f'$, we have  $f''=x_1y_2f_H\in \CC_r$.
If $m=2$, we are done. If $m>2$, $f_H$ will have factors
such as $(x_i-1)$ or $(x_i-1)(y_i-1)$, $i\geq 3$. We now describe group
ring operations on $f''$ which convert these factors to monomials.
By permuting the subscripts $1$, $2$ and $i$, we can obtain from  $f''$ a function which
starts with either $(x_1-1)x_2$ or $(x_1-1)(y_1-1)x_2$. We
calculate:
\begin{eqnarray*}
g_2(x_1-1)x_2&=&x_1x_2\\
(g_2-1)(x_1-1)(y_1-1)x_2&=&x_1(y_1-1)x_2\\
(g_2-1)(x_1-1)y_1x_2&=&x_1y_1x_2,
\end{eqnarray*}
where we interchanged $x_1 $ and $y_1$ in going from the second line
to the third line. By applying these operators to $f''$,
we obtain a new function in which the factors of $f_H$
involving the index $i$ have been replaced by monomials.
This process can be repeated for each $i\geq 3$
which occurs as an index in $f_H$. The resulting function
is the required monomial SBF.
\end{proof}

%%%%%%%%%%%%%%%%%%%%%%%%%%%%%%%%%%%%%%%%%%%%%%%%%%%%%%%%%%%%%%%%%%

\section{Proof of Theorem~\ref{rankthm}, Examples and Comparisons}
\label{examples}

We are ready to give the proof of Theorem~\ref{rankthm}.\\

\noindent{\bf Proof of Theorem~\ref{rankthm}.} By Theorem 5.2, the
$r$-admissible symplectic basis functions form a basis of $\CC_r$.
Hence
\begin{eqnarray}\label{intermediateform}
{\rm dim}(\CC_r) &=&1+{\mbox {number of non-constant $r$-admissible
SBFs}}\notag\\
&=&1+\sum_{(s_0,s_1\ldots ,s_{t-1})\leq (2m-r,\ldots
,2m-r)}\prod_{j=0}^{t-1} {\rm dim}(S_{\ell_j}^{\lambda_j}),
\end{eqnarray}
where $\lambda_j=2s_{j+1}-s_j$ and $\ell_j=(m-r)\delta(r\leq
m)+(2m-r-s_j)$. We use $a_{s_j,s_{j+1}}$ to denote ${\rm
dim}(S_{\ell_j}^{\lambda_j})$. By Lemma~\ref{wfilt}, we have
$$a_{s_j,s_{j+1}}={2m \choose 2s_{j+1}-s_j}-{2m \choose
2s_{j+1}-s_j-2\ell_j-2},$$ where $\ell_j$ is given as above. Let
$$A=(a_{i,j})_{1\leq i\leq 2m-r,\; 1\leq j\leq 2m-r}.$$
Then by (\ref{intermediateform}) we see that ${\rm dim}(\CC_r)$ is
equal to 1 plus the trace of $A^t$. The proof is complete.\qed
\vspace{0.1in}

As noted before, the $2$-rank formula for the incidence between
$\II_1$ and $\II_2$ when $m=2$ and $q=2^t$ is the same as what one
obtains using the formula obtained in \cite{CSX2} for the $p$-rank
of the incidence matrix of $\II_1$ and $\II_2$ when $m=2$ and
$q=p^t,\ p$ odd, and substituting $p=2$.  The same is not true when
$m>2$, except when $m=3$ and $q=2$.  We review the case for odd $q$.
\begin{Theorem}[\cite{CSX2}]
Let $V$ be a $2m$-dimensional vector space over $\F_q,\ q=p^t,\ p$
odd, equipped with a nondegenerate alternating bilinear form, and
let $C_{m,1}$ be the incidence matrix between totally isotropic
$m$-dimensional subspaces of $V$ and 1-dimensional subspaces of $V$.
Let $\lambda=pj-i$ and let
$$d_\lambda=\sum_{k=0}^{\lfloor\lambda/p\rfloor}(-1)^k{2m\choose k}{2m-1+\lambda-kp\choose 2m-1}.$$
Let $A'$ be the $m\times m$ matrix whose $(i,j)$-entry is
$$a'_{i,j}=\left\{
\begin{array}{lcr}
(d_\lambda+p^m)/2,\;&\mathrm{if}& i=j=m,\\
d_\lambda,\; &&\mathrm{otherwise}.
\end{array}\right.$$
Then
$$\rank_p(C_{r,1})=1+{\rm Trace}({A'}^t).$$
\end{Theorem}

Note that when $p=2$, the formula for $d_\lambda$ simplifies to
${2m\choose\lambda}$, and the matrix entry $a_{i,j}$ (for $r=m$) is
$d_\lambda-{2m\choose2j+i-2m-2}$.  The explanation that we get the
same formula for the $p=2$ and $p$ odd cases when $m=2$ is then that
the second term (i.e., ${2m\choose2j+i-2m-2}$) vanishes except when
$i=j=2$, in which case it is $1$, and $a_{2,2}=a'_{2,2}=5$.

We now compare the matrices $A$ and $A'$ when $m=3$.  Using the formulas, we get
$$A=\left(
\begin{array}{ccc}
6 &20&6\\1&15&14\\0&6&14\end{array}\right)\quad
A'=\left(
\begin{array}{ccc}
6 &20&6\\1&15&15\\0&6&14\end{array}\right).$$
Since the two matrices differ only off the diagonal, the actual rank will be less than what is given by the
$p$-odd model when $t>1$.  In fact, the eigenvalues of $A$ are $\alpha_1=8,\
\alpha_2=\frac{27}2+\frac{\sqrt{473}}2,$ and $\alpha_3=\frac{27}2-\frac{\sqrt{473}}2$, and the rank formula
may
be given as
$$\rank_2(B_{3,1}(t))=1+{\rm Trace}(A^t)=1+\alpha_1^t+\alpha_2^t+\alpha_3^t.$$
The $2$-ranks of $B_{3,1}(1),\ B_{3,1}(2),$ and $B_{3,1}(3)$ are $36,\ 666,$ and $15012,$ respectively.
The expressions for the eigenvalues of $A'$ are not so simple, and we do not reproduce them here.

When $m>3$, we note that $a_{m-1,m-1}<a'_{m-1,m-1}$.  Thus the actual $2$-rank of $B_{m,1}$ is less than
that obtained using the $p$-odd model, even when $q=2$.

We can also note that $a_{m,m}\le a'_{m,m}$.  The reason is that each function of class
$(\rho,0,2m-2\rho,0),\ \rho>0,$ is the sum of $2^{\rho-1}$  basis functions
of the $p$-odd model.  These
have the form
$(\prod_{i\in R}x_iy_i+\prod_{j\in R'}x_jy_j)\prod_{k\in T}z_k$, where $z_k=x_k$ or $z_k=y_k$, and $R$ and
$R'$ can be chosen $2^{\rho-1}$ different ways, fixing $r_1$ and $r'_1$, but interchanging $r_\gamma$ and
$r'_\gamma$ at will, for $2\le\gamma\le\rho$.  For $\rho=0$, the basis functions in the two models are
identical.

%%%%%%%%%%%%%%%%%%%%%%%%%%%%%%%%%%%%%%%%%%%%%%%%%%%%%%%%%%%%%%%%

\end{document}